\documentclass[twoside,11pt,english]{article}
\usepackage{blindtext}
\usepackage{jmlr2e}
\usepackage[latin9]{inputenc}
\usepackage[letterpaper]{geometry}
\usepackage{babel}
\usepackage{float}
\usepackage{units}

\usepackage{mathtools}
\usepackage{amsmath} 
\let\proof\relax

\usepackage{amsthm}
\usepackage{amssymb}

\usepackage{breqn} 

\usepackage{xcolor}
\usepackage{multirow}
\usepackage{colonequals}
\usepackage{bm}
\usepackage{bbm}
\usepackage{multicol}
\allowdisplaybreaks

\usepackage{cleveref}

\usepackage[noend]{algpseudocode}
\usepackage{algorithm}
\usepackage{algorithmicx}

\usepackage{bold-extra}
\renewcommand{\bf}{\bfseries}
\renewcommand{\sc}{\scshape}
\renewcommand{\it}{\itshape}

\providecommand{\lemmaname}{Lemma}
\providecommand{\remarkname}{Remark}
\providecommand{\theoremname}{Theorem}
\theoremstyle{plain}
\newtheorem{thm}{\protect\theoremname}
\theoremstyle{plain}
\newtheorem{lem}[thm]{\protect\lemmaname}
\theoremstyle{remark}
\newtheorem{rem}[thm]{\protect\remarkname}

\usepackage[backend=biber, style=numeric, maxcitenames=2, sorting=nty]{biblatex}
\usepackage{acro}
\DeclareAcronym{ABIDE}{
  short=ABIDE,
  long=Autism Brain Imaging Data Exchange,
  tag = abbrev
}
\DeclareAcronym{KL divergence}{
  short=KL divergence,
  long=Kullback--Leibler divergence,
  tag = abbrev
}
\DeclareAcronym{kNN}{
  short=$k$NN,
  long=$k$-Nearest Neighbors,
  tag = abbrev
}
\DeclareAcronym{DFT}{
  short=DFT,
  long=Discrete Fourier Transform,
  tag = abbrev
}
\DeclareAcronym{FFT}{
  short=FFT,
  long=Fast Fourier Transform,
  tag = abbrev
}
\DeclareAcronym{KDE}{
  short=KDE,
  long=Kernel Density Estimation,
  tag = abbrev
}
\DeclareAcronym{FFTKDE}{
  short=FFTKDE,
  long=Fast Fourier Transform-based Kernel Density Estimation,
  tag = abbrev
}
\DeclareAcronym{AG}{
  short=AG,
  long=Accelerated Gradient,
  tag = abbrev
}
\DeclareAcronym{CG}{
  short=CG,
  long=Conjugate Gradient,
  tag = abbrev
}
\DeclareAcronym{GMRES}{
  short=GMRES,
  long=Generalized Minimal Residuals,
  tag = abbrev
}

\DeclareAcronym{SCAD}{
  short=SCAD,
  long=Smoothly Clipped Absolute Deviation,
  tag = abbrev
}
\DeclareAcronym{MCP}{
  short=MCP,
  long=Minimax Concave Penalty,
  tag = abbrev
}

\DeclareAcronym{PCA}{
  short=PCA,
  long=Principal Components Analysis,
  tag = abbrev
}

\DeclareAcronym{ICA}{
  short=ICA,
  long=Independent Component Analysis,
  tag = abbrev
}

\DeclareAcronym{GWAS}{
  short=GWAS,
  long=Genome-Wide Association Studies,
  tag = abbrev
}

\DeclareAcronym{GMM}{
  short=GMM,
  long=Generalized Method of Moments,
  tag = abbrev
}

\DeclareAcronym{MRI}{
  short=MRI,
  long=Magnetic Resonance Images,
  tag = abbrev
}
\DeclareAcronym{CRLB}{
  short=CRLB,
  long=Cramer--Rao Lower Bound,
  tag = abbrev
}
\DeclareAcronym{NYSE}{
  short=NYSE,
  long=New York Stock Exchange,
  tag = abbrev
}
\DeclareAcronym{NASDAQ}{
  short=NSADAQ,
  long=National Association of Securities Dealers Automated Quotations,
  tag = abbrev
}
\DeclareAcronym{BSM}{
  short=BSM,
  long=Black--Scholes--Merton,
  tag = abbrev
}
\DeclareAcronym{LASSO}{
  short=LASSO,
  long=Least Absolute Shrinkage and Selection Operator,
  tag = abbrev
}
\DeclareAcronym{SNR}{
  short=SNR,
  long=Signal-to-Noise Ratio,
  tag = abbrev
}

\ShortHeadings{MTE Distributions for Robust and Efficient
Sparse Learning}{Yang, Asgharian, and Greenwood}
\firstpageno{1}

\begin{document}
\title{Maximum Tsallis Entropy Distributions for Robust and Efficient Sparse Learning from Correlated Data}
\author{\name Kai Yang \email kai.yang2@mail.mcgill.ca \\
       \AND
       \name Masoud Asgharian \email masoud.asgharian2@mcgill.ca \\
       \AND
       \name Celia M.T. Greenwood \email celia.greenwood@mcgill.ca}
\editor{My editor}

\maketitle
\begin{abstract}
This paper addresses the limitations of Gaussian distribution assumptions
in statistical sparse learning, particularly in modeling correlated
and heterogeneous data. Conventional Gaussian models often lack
robustness towards outliers and underlying distribution assumptions.
To overcome these limitations, we propose the use of the $q$Gaussian
distribution, derived from Tsallis entropy maximization, as a robust
alternative. This is notably relevant in biostatistics, where the presence of correlated observations and heterogeneity, such as in genetic and longitudinal studies, is prevalent.  Our contributions include modeling of correlated data through the re-derived multivariate probability density function from Tsallis entropy maximization, thereby addressing the limitations inherent in conventional Gaussian models. 
Furthermore, we introduce a novel framework that adapts numerical
methods designed to find equilibria in flows to tackle composite
optimization problems prevalent in statistical sparse learning. Applying
this framework to the Hager-Zhang conjugate gradient algorithm \cite{Hager2005}, we
develop a numerically stable and efficient algorithm for sparse statistical
learning. The $q$Gaussian distribution, informed by the principle
of maximizing Tsallis entropy, presents a viable and flexible alternative
to Gaussian-based methods. 
This paper not only contributes to
the theoretical understanding of statistical distributions and optimization
techniques, but also paves the way for practical data analysis.
\end{abstract}

\begin{keywords}
  Statistical Learning, high--dimensional, Entropy, Robust, Optimization 
\end{keywords}

\section{Introduction }

In the realm of statistical sparse learning, the pursuit of robust
and efficient methodologies remains paramount, especially when confronted
with the complexities of correlated data. The principle of maximizing
Shannon's entropy stands as a pivotal framework that has led to the
derivation of nearly all frequently utilized statistical distributions
to date \cite{Cover2006}. This principle's application has notably
revealed that the multivariate Gaussian distribution maximizes Shannon's
entropy under the first two moments, 
significantly influencing the landscape of statistical sparse learning.
The Gaussian assumption has become a fundamental 
cornerstone of numerous statistical sparse learning problem formulations, and its core assumptions are rarely re-examined or challenged. 

However, the Gaussian distribution's features, particularly its exponential
tail decay and the absence of a shape parameter, can present substantial
limitations. Specifically, the lack of robustness towards outliers
and a limited capacity to accurately represent the distribution's
shape results in violations of the Gaussian assumption in statistical
modeling. Such violations have many practical repercussions, including
the potential for erroneous Type I error rates and the lack of robustness
towards distribution shape when estimating the dispersion parameter,
motivating the development of alternative approaches.

Dispersion or volatility parameters encapsulate critical and often
decisive information about distributions. Their estimations, specifically
in transformations of predicted outcomes, are often indispensable.
For instance, in the context of log-normal distributions, the mean
is directly influenced by the volatility parameter derived from the
underlying Gaussian distribution. Likewise, principles like the Law
of the Unconscious Statistician (LOTUS), which rely on accurate estimation
of volatility and precise understanding of the distribution's shape,
highlight the importance of determining this parameter for dependable
prediction and statistical modeling. 

Volatility estimation is of great importance in finance. Specifically, It\^o's lemma, often used in stochastic calculus for option pricing, explicitly highlights the importance of volatility's contribution. Delving into the realm of stochastic calculus, It\^o's lemma provides
a mathematical framework that elegantly captures volatility's impact
on dynamic systems. It\^o's lemma states that for a twice-differentiable
function $f$, 
\begin{equation}
df\left(t,X_{t}\right)=\left(\frac{\partial f}{\partial t}+\mu\frac{\partial f}{\partial x}+\frac{1}{2}\sigma^{2}\frac{\partial^{2}f}{\partial x^{2}}\right)dt+\sigma\frac{\partial f}{\partial x}dW_{t},\label{eq:Itos-chain-rule}
\end{equation}
where the term $\sigma^{2}\frac{\partial^{2}f}{\partial x^{2}}$ specifically
denotes the contribution of volatility to changes in the function
$f$. This mathematical representation is pivotal in finance, where
the phenomenon, termed \emph{volatility smile}, challenges the foundational
assumptions of the \ac{BSM} model, signaling empirical
deviations from expected normality in option pricing models. These
deviations have propelled the exploration of alternative distributions
capable of more accurately reflecting market realities \cite{Pena1999}.

In response to these limitations of Gaussian distributions, the $q$Gaussian
distribution, derived from the maximization of the Tsallis entropy, emerges as a compelling
alternative. The $q$Gaussian distribution is celebrated for its
flexibility in modeling the diverse shapes of bell-curved distributions,
including the ability to account for heavy-tailed distributions ---
a feature crucial for robust modeling of financial returns. It
provides a more accurate representation of financial returns on platforms
such as the \emph{\ac{NYSE}} and \emph{\ac{NASDAQ}} \cite{Borland2002,Borland2002a,Domingo2017}.
Despite its proven advantages in finance, the incorporation of Tsallis
entropy-maximizing distributions within the domain of statistical
sparse learning and biostatistics remains limited. To the best of
our knowledge, this paper represents the initial endeavor to apply
Tsallis entropy-maximizing distributions for biostatistical data modeling.
Correlated observations, frequently encountered in genetic and longitudinal
studies \cite{Garcia2017,Runcie2019,DandineRoulland2015}, as well
as heterogeneity of the variance, 
will be specifically addressed by our proposed Tsallis entropy-maximizing model for correlated data.

Hence, this paper advocates for the application
of the $q$Gaussian distribution in modeling correlated data within
sparse statistical learning frameworks. Our approach relaxes the
conventional reliance on normality assumptions;  
We aim to demonstrate that the intricate characteristics of $q$Gaussian
distributions can profoundly enhance the modeling of correlated data,
offering a robust and versatile alternative to conventional Gaussian-based
methods. 

Maximum likelihood estimation is one of the most commonly used estimation
techniques. However, the estimation process encounters notable computational
obstacles when dealing with high--dimensional and extra-large datasets.
Oracle penalties, favored for their efficacy in facilitating variable
selection, present an attractive yet complex solution to sparse learning
problems. However, oracle penalties are notable for their nonconvex and nonsmooth nature \cite{Nikolova2000},
which lead to considerable optimization challenges. Recently, proximal
methods have demonstrated an unmatched speed of convergence, thereby
surpassing most other approaches in efficiently handling estimation
in nonsmooth problems \cite{Hoheisel2020}. Simultaneously, the Krylov
subspace method, recognized among the top ten algorithms for computing in science and engineering 
of the twentieth
century, lays a solid foundation for numerical analysis. The conjugate
gradient method, a prominent member of the Krylov subspace methods,
has been applied extensively in various areas and is a fundamental
numerical tool in solving partial differential equations \cite{Nocedal2000}. While the nonlinear conjugate gradient performs exceptionally well in terms of its convergence speed and numerical stability, much better than accelerated gradient or gradient descent, its global convergence depends on the line--search step, whereas accelerated gradient and gradient descent methods do not necessarily require the line search step to achieve global convergence \cite{Ghadimi2015, Yang2024}.
Motivated by these methodologies, our paper introduces a proximal conjugate
gradient method that can be applied to solve $q$Gaussian sparse learning problems. This method aims to effectively combine the theoretical strengths of both proximal methods and Krylov subspace techniques. Additionally, our paper addresses the line search step needed for the proximal nonlinear conjugate gradient method to establish global convergence. 

We re-derive the probability density function for the multivariate
$q$Gaussian distribution from a Tsallis entropy maximizing perspective in Lemma \ref{lem:density-normalization}.
This derivation allows for a nuanced understanding and application
of this model in statistical analysis. Furthermore, \emph{our contributions in this paper are as following: }
\begin{enumerate}
\item \emph{We apply the derived 
density to model correlated and heterogeneous data effectively,
while carrying out the sparse statistical learning at the same time. }
\item \emph{Sparse statistical learning involves minimizing a composite
optimization problem, aimed at minimizing a composite objective function
composed of a globally Lipschitz-smooth term, which may be nonconvex,
and a convex nonsmooth term. A variety of numerical methods are available
to find equilibrium points for globally Lipschitz flows. By employing
the Moreau envelope and linearizing the smooth term, we develop a
framework that allows any numerical method designed for finding equilibrium
points in globally Lipschitz flows to be adapted into a numerical
optimization algorithm for minimizing the composite objective function.} 
\item \emph{Leveraging the framework introduced above, we implement it with the state-of-the-art Hager-Zhang conjugate gradient method \cite{Hager2005}. This implementation yields a proximal conjugate gradient algorithm that is not only computationally efficient but also numerically stable, suitable for a wide range of statistical sparse learning challenges. This includes the robust sparse learning approach we devised based on the concept of maximizing the Tsallis entropy distribution. }
\end{enumerate}
The structure of the paper is organized as follows: 

Section \ref{sec:Tsallis-Entropy} delves into the foundational properties
of Tsallis entropy, drawing upon previous literature to establish
a comprehensive background. Following this, Section \ref{sec:Tsallis-Entropy-Maximizing}
introduces the concept of $q-$moments. This section then elaborates
on employing Tsallis entropy maximizing distribution to effectively
model the $q-$correlation structure. In Section \ref{sec:Tsallis-Entropy-Maximizing},
we also re-derive the probability density function maximizing Tsallis
entropy under the first and second central $q-$moment constraints,
incorporating all relevant parameters for a likelihood-based approach
to statistical analysis. 

Our discussion transitions to the challenges and strategies of optimization
in Section \ref{sec:optimization}. This section is twofold; initially,
in Section \ref{subsec:variational-analysis}, we present essential background
knowledge from variational and nonsmooth analysis. This foundation
is critical for our novel contribution: the development of a proximal
framework to transform any first-order numerical optimization algorithm
to a proximal counterpart by leveraging the properties of the Moreau
envelope, detailed in Section \ref{subsec:PCG-framework}. In Section \ref{subsec:Proximal-HZ},
we apply this innovative framework to the state-of-the-art Hager-Zhang
conjugate gradient algorithm. This adaptation produces a proximal
version for tackling sparse statistical learning challenges. The efficacy
of this method is further showcased in Section \ref{sec:optimizing-penalized-mle},
where we outline the application of our proximal Hager-Zhang conjugate
gradient algorithm to optimize a penalized $q$Gaussian likelihood
function. This section also lays out a map from problem formulation
to the practical aspects of prediction using models trained with our
approach.
Finally, Section \ref{sec:Conclusion-and-Discussion}
synthesizes our contributions, offering a reflective conclusion and
proposing avenues for future research. 

\section{\label{sec:Tsallis-Entropy} Tsallis Entropy }

For an arbitrary random variable $X$, Shannon's Entropy \cite{Shannon1948}
poses the definition 
\begin{equation}
H\left(X\right)\coloneqq-\mathbb{E}\log\left(p\left(X\right)\right)=-\int\log\left(p\left(x\right)\right)d\mu_{X},\label{eq:shannon-entropy}
\end{equation}
where $p$ is the likelihood function for $X$. Over a given (likelihood)
function space 
\begin{equation}
    \mathcal{P}\coloneqq\left\{ p\left(x\right)\vert\forall x\in\mathcal{X},\ p\left(x\right)\geq0,\left|\mathbb{E}\log\left(p\left(X\right)\right)\right|<\infty\text{ and }\int_{\mathcal{X}}1d\mu_{X}=1\right\} ,
    \notag
\end{equation}
the Shannon's entropy is a \emph{strictly concave} function, which
implies uniqueness of the maximizer. Many commonly-used distributions
have been shown to maximize Shannon's entropy under certain given
constraints \cite{Cover2006}. For example, uniform distribution,
whether in discrete or continuous case
, are to maximize \eqref{eq:shannon-entropy} over a compact support,
with open sets defined by discrete or Euclidean topology, respectively.
The exponential distribution is defined as maximizing \eqref{eq:shannon-entropy}
over $\mathbb{R}_{\geq0}$ and with a constraint that the first moment
is a constant, $\frac{1}{\lambda}$; where $\lambda$ later turns
out to be the scale parameter. And the Gaussian distribution maximizes
\eqref{eq:shannon-entropy} over $\mathbb{R}$ with given mean and
variance. More examples can be given. For example, the constraint
to obtain a Laplace distribution is a given mean absolute deviation,
etc. 

\emph{Additivity} is a key element of Shannon's entropy. That is,
let $A_{1},A_{2}$ be two independent event sets, then the information
of the intersection $I\left(\mathbb{P}\left(A_{1}\cap A_{2}\right)\right)=I\left(\mathbb{P}\left(A_{1}\right)\cdot\mathbb{P}\left(A_{2}\right)\right)=I\left(\mathbb{P}\left(A_{1}\right)\right)+I\left(\mathbb{P}\left(A_{2}\right)\right)$
--- such homomorphism was considered particularly useful in Shannon's
view \cite{Shannon1948}. 
Later in the 1980s, \citeauthor{Tsallis1988} constructed an entropy
similar to Shannon's entropy but without the additivity property.
To see how Tsallis' entropy was developed, first we look at Tsallis'
$q-$exponential function, which is defined as $\exp_{q}:\mathbb{R}\mapsto\mathbb{R}$,
given by 
\begin{equation}
\exp_{q}x\coloneqq\begin{cases}
\left(\left(1+\left(1-q\right)x\right)\right)^{\frac{1}{1-q}}; & 1+\left(1-q\right)x>0\\
0; & \text{else}
\end{cases}\label{eq:tsalli-q-exp}
\end{equation}
For $q>1$, $\exp_{q}$ is bijective over $\left(0,\frac{1}{q-1}\right)$.
The inverse function, called the $q-$logarithmic function, is given
by 
\begin{equation}
\ln_{q}x\coloneqq\frac{x^{1-q}-1}{1-q}.\label{eq:tsalli-q-log}
\end{equation}
Based on this deformed $q-$exponential function, \cite{Tsallis1988}
developed \emph{Tsallis entropy} by replacing the $\log$ function
in \ref{eq:shannon-entropy} with $q-\log$ function \eqref{eq:tsalli-q-log}
and replacing the expectation with $q-$expectation operator \cite{Tsallis1988}:
\begin{align}
S_{q}\left(X\right) & =-\int_{\mathcal{X}}p^{q}\left(x\right)\ln_{q}p\left(x\right)dx\eqqcolon-\mathbb{E}_{q}\ln_{q}p\left(X\right)\label{eq:tsallis-entropy}\\
 & =\frac{1}{q-1}\left(1-\int_{\mathcal{X}}p^{q}\left(x\right)dx\right)
\end{align}
where $q\in\mathbb{R}\setminus\{1\}$ is a constant, and $\mathbb{E}_{q}f\left(X\right)\coloneqq\int_{\mathcal{X}}f\left(x\right)\cdot p^{q}\left(x\right)dx=\left\langle f\left(x\right),\left(d\mu_{X}\right)^{q}\right\rangle $
is referred to as the \emph{$q-$expectation operator}. Tsallis entropy
is also known as \emph{non-extensive} entropy; namely for arbitrary
independent two random variable $X_{1},X_{2}$: 
\begin{equation}
S_{q}(X_{1},X_{2})=S_{q}(X_{1})\oplus_{q}S_{q}(X_{2}),\label{eq:q-nonextensive}
\end{equation}
where ``$\oplus_{q}$'' is defined as $\forall a,b\in\mathbb{R}$,
\begin{equation}
a\oplus_{q}b\coloneqq a+b+\left(1-q\right)ab.
\end{equation}
Expectation has been used to characterize statistical distributions.
However, one significant drawback of the expectation (linear) operator
is the lack of continuity for some distributions; such as the Cauchy
distribution. Therefore, the $q-$expectation operator, $\mathbb{E}_{q}$,
provides robustness when characterizing the distributions in the real
domain. If the tail of the function vanishes at a rate of $O\left(\left(\log x\right)^{-1}\right)$, the function will not have a proper integral if the support
is unbounded. Thus, for any distribution whose likelihood function
is bounded in uniform norm, $\exists q\in\mathbb{R}_{>0}$ such that
$\mathbb{E}_{q}$ is continuous at the likelihood function in the
function space we are considering. 

\section{\label{sec:Tsallis-Entropy-Maximizing} Tsallis Entropy Maximizing
Distribution to Accommodate the $q-$Correlation Structure }

The Gaussian distribution maximizes the Shannon's entropy in the following
problem: 
\begin{align}
\max\ _{\phi\in \mathcal{P}} & -\int_{\mathbb{R}^{n}}\phi\left(x\right)\log\left(\phi\left(x\right)\right)dx\nonumber \\
\text{s.t. } & \phi\geq0;\nonumber \\
 & \int_{\mathbb{R}^{n}}\phi\left(x\right)dx=1;\nonumber \\
 & \int_{\mathbb{R}^{n}}x\cdot\phi\left(x\right)dx=0;\label{eq:shannon-gaussian-location}\\
 & \int_{\mathbb{R}^{n}}xx^{T}\cdot\phi\left(x\right)dx=\Sigma;\nonumber 
\end{align}
for some $n\in\mathbb{N}_{+}$ and $\Sigma\in\mathbb{R}^{n\times n},\ \Sigma\succ0$. For the sake of parsimony,
in \eqref{eq:shannon-gaussian-location} we assume that the distribution
is centered. To set the central trend parameter, or the mean parameter
in the specific case of the Gaussian distribution, the likelihood function
$\phi$ can be simply translated $x\mapsto x-\mu$ to incorporate
the parameter $\mu$ for the central trend. Entropy functions are invariant
under translation. 

Similarly to how the multivariate Gaussian distribution maximizes Shannon's
entropy in a Euclidean space, the multivariate $q$Gaussian distribution maximizes Tsallis entropy in a Euclidean space. Specifically, the
optimization problem is formulated as: 
\begin{align}
\max\ _{\phi\in L^{q}\left(\mathbb{R}^{n}\right)} & -\int_{\mathbb{R}^{n}}\phi^{q}\left(x\right)dx\label{eq:generalized-entropy-obj}\\
\text{s.t. } & \phi\geq0;\nonumber \\
 & \int_{\mathbb{R}^{n}}\phi\left(x\right)dx=1;\label{eq:normalization-constraint}\\
 & \frac{\int_{\mathbb{R}^{n}}x\cdot\phi^{q}\left(x\right)dx}{\int_{\mathbb{R}^{n}}\phi^{q}\left(x\right)dx}=0;\label{eq:q-first-moment-constraint}\\
 & \frac{\int_{\mathbb{R}^{n}}xx^{T}\cdot\phi^{q}\left(x\right)dx}{\int_{\mathbb{R}^{n}}\phi^{q}\left(x\right)dx}=\Sigma,\label{eq:q-second-moment-constraint}
\end{align}
where $q>1$. The feasible set of Lebesgue space $L^{q}\left(\mathbb{R}^{n}\right)$ is to ensure the well--definedness of Tsallis entropy. The normalization constraint \eqref{eq:normalization-constraint} implies that $\phi\in L^{1}$; however, $\phi\in L^{1}$ does not imply $\phi\in L^{q}$, as the embedding property $L^{1} \subseteq L^{q}$ fails to hold for Lebesgue measure on $\mathbb{R}^{n}$. As an example, consider the one-dimensional example of the probability density function 
\begin{equation}
    \tilde{\phi}\left(x\right)=\begin{cases}
\frac{1}{4}\left|x\right|^{-\frac{1}{2}} & \text{for }x\in\left(-1,1\right)\setminus\left\{ 0\right\} ;\\
0 & \text{else.}
\end{cases}
\end{equation}
Clearly, $\tilde{\phi}\in L^{1}$ but $\tilde{\phi}\not\in L^{2}$. Note that \eqref{eq:q-first-moment-constraint} and
\eqref{eq:q-second-moment-constraint} are the first and second moment
constraints using the $q-$expectation operator $\mathbb{E}_{q}$. As
noted by \cite{Tsukada2005}, maximizing any member of the generalized
class of power-law entropies, including Renyi entropy, Havrda and
Charvat entropy, Arimoto entropy, and Tsallis entropy, all yield the
identical power-law objective function \eqref{eq:generalized-entropy-obj}.
Regarding the constraints, $\int_{\mathbb{R}^{n}}\phi^{q}\left(x\right)dx$
is the normalization factor for the $q-$expectation. \cite{Tsukada2005}
further noted that optimizing the problem formulated above is equivalent
to the following problem: 
\begin{align}
\max\ _{\varphi\in L^{s}\left(\mathbb{R}^{n}\right)} & \int_{\mathbb{R}^{n}}\varphi^{s}\left(x\right)dx\label{eq:generalized-entropy-moment-obj}\\
\text{s.t. } & \varphi\geq0;\nonumber \\
 & \int_{\mathbb{R}^{n}}\varphi\left(x\right)dx=1;\nonumber \\
 & \int_{\mathbb{R}^{n}}x\cdot\varphi\left(x\right)dx=0;\nonumber \\
 & \int_{\mathbb{R}^{n}}xx^{T}\cdot\varphi\left(x\right)dx=\Sigma.\nonumber 
\end{align}
In \eqref{eq:generalized-entropy-moment-obj}, $s\coloneqq q^{-1}\in\left(0,1\right)$, thus $L^{s}\left(\mathbb{R}^{n}\right)$ is a quasi-normed space;
$\varphi\left(x\right)\coloneqq\frac{\phi^{q}\left(x\right)}{\int_{\mathbb{R}^{n}}\phi^{q}\left(x\right)dx}.$
If the maximizer of \eqref{eq:generalized-entropy-moment-obj} is
$\varphi$, then the maximizer of \eqref{eq:generalized-entropy-obj},
$\phi$, will be normalized 
\begin{equation}
\phi\left(x\right)\propto\varphi^{1/q}.
\end{equation}
Several important properties were proposed previously
regarding the $q$Gaussian distributions in previous studies \cite{Vignat2004,Costa2003}.
Notably, 
\begin{enumerate}
\item \label{enu:maximizing-format} Using Bregman information divergence,
Problem \eqref{eq:generalized-entropy-moment-obj} has a unique maximizer
of the form 
\begin{equation}
\varphi\left(x;s\right)=A_{s}\left(1-\left(s-1\right)\beta^{\prime}\left\langle x,\Sigma^{-1}x\right\rangle \right)_{+}^{\frac{1}{s-1}}\label{eq:dual-maximizer}
\end{equation}
for some $s\in\left(\frac{n}{n+2},\infty\right)\setminus\left\{ 1\right\} $,
normalization constant $A_{r}$, and some dispersion parameter $\beta^{\prime}$. 
\item \label{enu:linear-mapping-invariance} If $X\sim q\text{Gaussian}\left(q,\Sigma\right)$,
$H\in\mathbb{R}^{\tilde{n}\times n}$ and $\text{rank}\left(H\right)=\tilde{n}$.
Then $\tilde{X}\sim q\text{Gaussian}\left(\tilde{q},H\Sigma H^{T}\right)$
with 
\begin{equation}
\frac{2}{1-\tilde{q}^{-1}}-\tilde{n}=\frac{2}{1-q^{-1}}-n.\label{eq:linear-mapping-invariance}
\end{equation}
 
\item \label{enu:linear-combination-not-closed} If $X_{1},X_{2}$ are both
$q$Gaussian random vectors but independent, a linear combination
of $H_{1}X_{1}+H_{2}X_{2}$ is {\em not} $q$Gaussian. 
\item \label{enu:heavy-tail-bounded-support-duality} The duality property:
if $X\sim q\text{Gaussian}\left(q,\Sigma\right)$ with $1<q<1+\frac{2}{n}$,
let the degree of freedom for $X$ be $m\coloneqq\frac{2}{q-1}-n$
and $\Lambda\coloneqq m\Sigma$, then 
\begin{align*}
\frac{X}{\sqrt{1-\left\langle X,\Lambda^{-1}X\right\rangle }} & \sim q\text{Gaussian}\left(\tilde{q},\frac{m}{m+4}\Sigma\right)\\
\text{with }\frac{1}{\tilde{q}^{-1}-1} & =\frac{1}{1-q^{-1}}-\frac{n}{2}-1,
\end{align*}
and $0<\tilde{q}<1$. 
\end{enumerate}
Property \ref{enu:maximizing-format} will be used in our Lemma \ref{lem:density-normalization}.
Property \ref{enu:linear-mapping-invariance} implies that any components of a $q$Gaussian random vector are also $q$Gaussian, while Property \ref{enu:linear-combination-not-closed}
implies that two independent $q$Gaussian vectors are not jointly
$q$Gaussian. 

By the equivalence of problems \eqref{eq:generalized-entropy-obj}
and \eqref{eq:generalized-entropy-moment-obj} discussed before, \eqref{eq:dual-maximizer}
can be rewritten as 
\begin{equation}
\phi\left(x;q,\Sigma\right)=\left(\alpha-\beta\left\langle x,\Sigma^{-1}x\right\rangle \right)_{+}^{\frac{1}{1-q}}\label{eq:parameter-to-be-found}
\end{equation}
for some constant (parameter) $\alpha,\beta$, $q\in\left(0,1+\frac{2}{n}\right)\setminus\left\{ 1\right\} $;
$x_{+}\coloneqq\max\ \left(0,x\right)$. As shown later in the proof
of Lemma \ref{lem:density-normalization}, the dimension-related
upper bound $1+\frac{2}{n}$ is due to the normalization constraint
\eqref{eq:normalization-constraint}. When $0<q<1$, the density represents
a distribution with bounded support; when $q>1$, the density is a
generalization of the bell curve distributions, and with $q\searrow1$ the
Gaussian distribution is recovered. A higher value of $q$ corresponds
to heavier tails in shape. The duality between the $q$Gaussian random vectors
with $0<q<1$ and $1<q<1+\frac{2}{n}$ was given by \cite{Vignat2004},
which we discussed in Property \ref{enu:heavy-tail-bounded-support-duality} in Section \ref{sec:Tsallis-Entropy-Maximizing}.
Distributions with bounded support correspond to $0<q<1$, and distributions
with heavy tails correspond to $1<q<1+\frac{2}{n}$. For the scope
of this paper, we will focus only on the heavy-tail distributions;
i.e., the case when $q>1$. \cite{Vignat2009} derived the $q$Gaussian
probability density function for $1<q<\frac{n+4}{n+2}$, when the
multivariate $q$Gaussian density becomes the scaled density of the multivariate
student's $t$ distribution. To incorporate the case of $q\in[1+\frac{2}{n+2},1+\frac{2}{n})$,
when the variance does not exist but the $q-$variance can be used
to capture the volatility/dispersion of the data, \cite{Vignat2007b}
also derived the resulting density; however, since a typo was found
in that paper, we re-derive the density in Lemma \ref{lem:density-normalization}.
The parameters presented in the density formula \eqref{eq:$q$Gaussian-heavy-tail-density}
are of particular interest to statisticians, as parameter inference
is the key to statistical analysis and prediction. The case of $q\in[1+\frac{2}{n+2},1+\frac{2}{n})$
will allow the resulting $q$Gaussian distribution to incorporate
the wider class of distributions without finite moments but finite
$q-$moments; such as the Cauchy distribution. Therefore, modeling using the $q$Gaussian distribution with $q$ allowed to take the value in $[1+\frac{2}{n+2},1+\frac{2}{n})$
will be more robust. 
\begin{lem}
\label{lem:density-normalization} When $q\in\left(1,1+\frac{2}{n}\right)$, the unique solution to \eqref{eq:generalized-entropy-obj}
is: 
\begin{equation}
p\left(x;q,\Sigma\right)=\frac{1}{\left|\pi\Sigma\right|^{\nicefrac{1}{2}}}\cdot\frac{\Gamma\left(\frac{1}{q-1}\right)}{\Gamma\left(\frac{1}{q-1}-\frac{n}{2}\right)}\cdot\left(\frac{2}{q-1}-n\right)^{-\frac{n}{2}}\cdot\left(1+\left(\frac{2}{q-1}-n\right)^{-1}\cdot\left\langle x,\Sigma^{-1}x\right\rangle \right)^{\frac{1}{1-q}}.\label{eq:$q$Gaussian-heavy-tail-density}
\end{equation}
\end{lem}

\begin{proof}
By \eqref{eq:dual-maximizer} and the equivalence of the problems \eqref{eq:generalized-entropy-obj}
and \eqref{eq:generalized-entropy-moment-obj}, let the solution to
\eqref{eq:generalized-entropy-obj} be denoted by 
\begin{equation}
p\left(x;q,\Sigma\right)=\frac{1}{Z}\left(\gamma+\left\langle x,\Sigma^{-1}x\right\rangle \right)^{\frac{1}{1-q}}\label{eq:target-maximizer}
\end{equation}
for some $Z,\gamma>0$. Feasibility for problem \ref{eq:generalized-entropy-obj} when $q\in\left(1,1+\frac{2}{n}\right)$
was given in \cite{Vignat2004}. Hence, the strictly concavity of
the objective function \ref{eq:generalized-entropy-obj} implies that the optimal solution is unique. The symmetry of $p\left(x;q,\Sigma\right)$
is implied by \eqref{eq:target-maximizer}; thus, we reformulate the
problem \eqref{eq:generalized-entropy-obj} as the following equivalent
problem: 
\begin{align}
\max\ _{p\in L^{q}\left(\mathbb{R}^{n}\right)} & -\int_{\mathbb{R}^{n}}\left(p\left(x;q,\Sigma\right)\right)^{q}dx\nonumber \\
\text{s.t. } & p\left(x;q,\Sigma\right)\geq0;\nonumber \\
 & \int_{\mathbb{R}_{>0}^{n}}p\left(x;q,\Sigma\right)dx=2^{-n};\nonumber \\
 & \frac{\int_{\mathbb{R}^{n}}x\cdot p^{q}\left(x;q,\Sigma\right)dx}{\int_{\mathbb{R}^{n}}p^{q}\left(x;q,\Sigma\right)dx}=0;\nonumber \\
 & \frac{\int_{\mathbb{R}^{n}}xx^{T}\cdot p^{q}\left(x;q,\Sigma\right)dx}{\int_{\mathbb{R}^{n}}p^{q}\left(x;q,\Sigma\right)dx}=\Sigma.\label{eq:second-q-moment-constraint}
\end{align}
Thus, 
\begin{align}
Z & =2^{n}\int_{\mathbb{R}_{>0}^{n}}\left(\gamma+\left\langle x,\Sigma^{-1}x\right\rangle \right)^{\frac{1}{1-q}}dx\nonumber \\
 & =2^{n}\left|\Sigma^{\nicefrac{1}{2}}\right|\int_{\mathbb{R}_{>0}^{n}}\left(\gamma+\left\langle x,x\right\rangle \right)^{\frac{1}{1-q}}dx\nonumber \\
 & =2^{n}\left|\Sigma\right|^{\nicefrac{1}{2}}\int_{0}^{\infty}r^{n-1}\left(\gamma+r^{2}\right)^{\frac{1}{1-q}}\left(\prod_{i=1}^{n-2}\int_{0}^{\frac{\pi}{2}}\sin^{n-1-i}(\theta_{i})d\theta_{i}\cdot\int_{0}^{\frac{\pi}{2}}1d\theta\right)dr\nonumber \\
 & =2^{n}\left|\Sigma\right|^{\nicefrac{1}{2}}\cdot\left(\prod_{i=1}^{n-2}\int_{0}^{\frac{\pi}{2}}\sin^{n-1-i}(\theta_{i})d\theta_{i}\cdot\int_{0}^{\frac{\pi}{2}}1d\theta\right)\cdot\int_{0}^{\infty}r^{n-1}\left(\gamma+r^{2}\right)^{\frac{1}{1-q}}dr\nonumber \\
 & =\pi2^{n-1}\left|\Sigma\right|^{\nicefrac{1}{2}}\cdot\left(\prod_{i=1}^{n-2}\frac{1}{2}\frac{\Gamma\left(\frac{n-i}{2}\right)\sqrt{\pi}}{\Gamma\left(\frac{n-i+1}{2}\right)}\right)\cdot\int_{0}^{\infty}r^{n-1}\left(\gamma+r^{2}\right)^{\frac{1}{1-q}}dr\nonumber \\
 & =2\pi^{\frac{n}{2}}\left|\Sigma\right|^{\nicefrac{1}{2}}\cdot\left(\Gamma\left(\frac{n}{2}\right)\right)^{-1}\cdot\int_{0}^{\infty}r^{n-1}\left(\gamma+r^{2}\right)^{\frac{1}{1-q}}dr\nonumber \\
 & =2\pi^{\frac{n}{2}}\left|\Sigma\right|^{\nicefrac{1}{2}}\cdot\left(\Gamma\left(\frac{n}{2}\right)\right)^{-1}\cdot\int_{0}^{\infty}\gamma^{\frac{n-1}{2}+\frac{1}{1-q}}\left(\frac{r}{\sqrt{\gamma}}\right)^{n-1}\left(1+\left(\frac{r}{\sqrt{\gamma}}\right)^{2}\right)^{\frac{1}{1-q}}dr\nonumber \\
 & =2\pi^{\frac{n}{2}}\left|\Sigma\right|^{\nicefrac{1}{2}}\cdot\left(\Gamma\left(\frac{n}{2}\right)\right)^{-1}\cdot\gamma^{\frac{n}{2}+\frac{1}{1-q}}\cdot\int_{0}^{\infty}\left(r^{\prime}\right)^{n-1}\left(1+\left(r^{\prime}\right)^{2}\right)^{\frac{1}{1-q}}dr^{\prime}\nonumber \\
 & =2\pi^{\frac{n}{2}}\left|\Sigma\right|^{\nicefrac{1}{2}}\cdot\left(\Gamma\left(\frac{n}{2}\right)\right)^{-1}\cdot\gamma^{\frac{n}{2}+\frac{1}{1-q}}\cdot\int_{0}^{\infty}\left(\left(r^{\prime}\right)^{1-n}\left(1+\left(r^{\prime}\right)^{2}\right)^{\frac{1}{q-1}}\right)^{-1}dr^{\prime}\nonumber \\
 & =2\pi^{\frac{n}{2}}\left|\Sigma\right|^{\nicefrac{1}{2}}\cdot\left(\Gamma\left(\frac{n}{2}\right)\right)^{-1}\cdot\gamma^{\frac{n}{2}+\frac{1}{1-q}}\cdot\frac{1}{2}B\left(\frac{1}{q-1}-\frac{n}{2},\frac{n}{2}\right)\label{eq:beta-function-step}\\
 & =\pi^{\frac{n}{2}}\left|\Sigma\right|^{\nicefrac{1}{2}}\cdot\left(\Gamma\left(\frac{n}{2}\right)\right)^{-1}\cdot\gamma^{\frac{n}{2}+\frac{1}{1-q}}\cdot\frac{\Gamma\left(\frac{1}{q-1}-\frac{n}{2}\right)\Gamma\left(\frac{n}{2}\right)}{\Gamma\left(\frac{1}{q-1}\right)}\nonumber \\
 & =\pi^{\frac{n}{2}}\left|\Sigma\right|^{\nicefrac{1}{2}}\frac{\Gamma\left(\frac{1}{q-1}-\frac{n}{2}\right)}{\Gamma\left(\frac{1}{q-1}\right)}\cdot\gamma^{\frac{n}{2}+\frac{1}{1-q}}.\nonumber 
\end{align}
In step \eqref{eq:beta-function-step}, we use the following formula
for Beta function \cite{Tsukada2005}: 
\begin{equation}
\int_{0}^{\infty}\left(x^{\alpha}\left(1+x^{\lambda}\right)^{\beta}\right)^{-1}dx=\frac{1}{\lambda}B\left(\beta-\frac{1-\alpha}{\lambda},\frac{1-\alpha}{\lambda}\right),\label{eq:beta-function-formula}
\end{equation}
where $\alpha<1,\ \lambda>0,\ \beta>0,\ \lambda\beta>1-\alpha$. Well--definedness of $Z$ and \eqref{eq:beta-function-step} implies that $\frac{1}{q-1}-\frac{n}{2}>0$; i.e.,
$q<1+\frac{2}{n}$, which is the reason for the upper bound for the
choice of $q$. 

\eqref{eq:second-q-moment-constraint} implies that 
\begin{equation}
\text{tr}\left(\Sigma^{-1}\frac{\int_{\mathbb{R}^{n}}xx^{T}\cdot p^{q}\left(x\right)dx}{\int_{\mathbb{R}^{n}}p^{q}\left(x\right)dx}\right)=\text{tr}\left(\Sigma^{-1}\Sigma\right)=n.\label{eq:second-q-moment-trace-trick-1}
\end{equation}
Hence, since $p\left(x\right)$ is symmetric, 
\begin{align}
& \text{tr}\left(\Sigma^{-1}\frac{\int_{\mathbb{R}^{n}}xx^{T}\cdot p^{q}\left(x;q,\Sigma\right)dx}{\int_{\mathbb{R}^{n}}p^{q}\left(x;q,\Sigma\right)dx}\right) \\ & =\text{tr}\left(\Sigma^{-1}\frac{\int_{\mathbb{R}_{>0}^{n}}xx^{T}\cdot p^{q}\left(x;q,\Sigma\right)dx}{\int_{\mathbb{R}_{>0}^{n}}p^{q}\left(x;q,\Sigma\right)dx}\right)\nonumber \\
 & =\frac{\text{tr}\left(\int_{\mathbb{R}_{>0}^{n}}\Sigma^{-1}xx^{T}\cdot p^{q}\left(x;q,\Sigma\right)dx\right)}{\int_{\mathbb{R}_{>0}^{n}}p^{q}\left(x;q,\Sigma\right)dx}\nonumber \\
 & =\frac{\int_{\mathbb{R}_{>0}^{n}}\text{tr}\left(\Sigma^{-1}xx^{T}\cdot p^{q}\left(x;q,\Sigma\right)\right)dx}{\int_{\mathbb{R}_{>0}^{n}}p^{q}\left(x;q,\Sigma\right)dx}\nonumber \\
 & =\frac{\int_{\mathbb{R}_{>0}^{n}}\text{tr}\left(x^{T}\Sigma^{-1}x\right)\cdot p^{q}\left(x;q,\Sigma\right)dx}{\int_{\mathbb{R}_{>0}^{n}}p^{q}\left(x;q,\Sigma\right)dx}\nonumber \\
 & =\frac{\int_{\mathbb{R}_{>0}^{n}}\left\langle x,\Sigma^{-1}x\right\rangle \cdot\left(\frac{1}{Z}\left(\gamma+\left\langle x,\Sigma^{-1}x\right\rangle \right)^{\frac{1}{1-q}}\right)^{q}dx}{\int_{\mathbb{R}_{>0}^{n}}\left(\frac{1}{Z}\left(\gamma+\left\langle x,\Sigma^{-1}x\right\rangle \right)^{\frac{1}{1-q}}\right)^{q}dx}\nonumber \\
 & =\frac{\int_{\mathbb{R}_{>0}^{n}}\left\langle x,\Sigma^{-1}x\right\rangle \cdot\left(\gamma+\left\langle x,\Sigma^{-1}x\right\rangle \right)^{\frac{q}{1-q}}dx}{\int_{\mathbb{R}_{>0}^{n}}\left(\gamma+\left\langle x,\Sigma^{-1}x\right\rangle \right)^{\frac{q}{1-q}}dx}\nonumber \\
 & =\frac{\int_{\mathbb{R}_{>0}^{n}}\left|\Sigma\right|^{\nicefrac{1}{2}}\left\langle x,x\right\rangle \cdot\left(\gamma+\left\langle x,x\right\rangle \right)^{\frac{q}{1-q}}dx}{\int_{\mathbb{R}_{>0}^{n}}\left|\Sigma\right|^{\nicefrac{1}{2}}\left(\gamma+\left\langle x,x\right\rangle \right)^{\frac{q}{1-q}}dx}\nonumber \\
 & =\frac{\int_{\mathbb{R}_{>0}^{n}}\left\langle x,x\right\rangle \cdot\left(\gamma+\left\langle x,x\right\rangle \right)^{\frac{q}{1-q}}dx}{\int_{\mathbb{R}_{>0}^{n}}\left(\gamma+\left\langle x,x\right\rangle \right)^{\frac{q}{1-q}}dx}\nonumber \\
 & =\frac{\int_{0}^{\infty}r^{n-1}\cdot r^{2}\cdot\left(\gamma+r^{2}\right)^{\frac{q}{1-q}}\cdot\left(\prod_{i=1}^{n-2}\int_{0}^{\frac{\pi}{2}}\sin^{n-1-i}(\theta_{i})d\theta_{i}\cdot\int_{0}^{\frac{\pi}{2}}1d\theta\right)dr}{\int_{0}^{\infty}r^{n-1}\cdot\left(\gamma+r^{2}\right)^{\frac{q}{1-q}}\cdot\left(\prod_{i=1}^{n-2}\int_{0}^{\frac{\pi}{2}}\sin^{n-1-i}(\theta_{i})d\theta_{i}\cdot\int_{0}^{\frac{\pi}{2}}1d\theta\right)dr}\nonumber \\
 & =\frac{\int_{0}^{\infty}r^{n+1}\cdot\left(\gamma+r^{2}\right)^{\frac{q}{1-q}}dr}{\int_{0}^{\infty}r^{n-1}\cdot\left(\gamma+r^{2}\right)^{\frac{q}{1-q}}dr}\nonumber \\
 & =\frac{\gamma\int_{0}^{\infty}\left(\left(r^{\prime}\right)^{-n-1}\cdot\left(1+\left(r^{\prime}\right)^{2}\right)^{\frac{q}{q-1}}\right)^{-1}dr^{\prime}}{\int_{0}^{\infty}\left(\left(r^{\prime}\right)^{1-n}\cdot\left(1+\left(r^{\prime}\right)^{2}\right)^{\frac{q}{q-1}}\right)^{-1}dr^{\prime}}\nonumber \\
 & =\frac{\gamma B\left(\frac{q}{q-1}-\frac{n+2}{2},\frac{n+2}{2}\right)}{B\left(\frac{q}{q-1}-\frac{n}{2},\frac{n}{2}\right)}\label{eq:beta-function-usage-2}\\
 & =\gamma\cdot\frac{\Gamma\left(\frac{q}{q-1}-\frac{n}{2}-1\right)\Gamma\left(\frac{n+2}{2}\right)}{\Gamma\left(\frac{q}{q-1}\right)}/\frac{\Gamma\left(\frac{q}{q-1}-\frac{n}{2}\right)\Gamma\left(\frac{n}{2}\right)}{\Gamma\left(\frac{q}{q-1}\right)}\nonumber \\
 & =\gamma\cdot\frac{n}{2}\cdot\left(\frac{q}{q-1}-\frac{n}{2}-1\right)^{-1}.\label{eq:second-q-moment-trace-trick-2}
\end{align}
In step \eqref{eq:beta-function-usage-2}, we used \eqref{eq:beta-function-formula}.
Combining \eqref{eq:second-q-moment-trace-trick-1} and \eqref{eq:second-q-moment-trace-trick-2},
we have 
\begin{equation}
\gamma\cdot\frac{n}{2}\cdot\left(\frac{q}{q-1}-\frac{n}{2}-1\right)^{-1}=n,
\end{equation}
which gives that 
\begin{equation}
\gamma=\frac{2q}{q-1}-n-2=\frac{2}{q-1}-n.
\end{equation}
Thus, the probability density function that maximizes problem \eqref{eq:generalized-entropy-obj}
is: 
\begin{align*}
p\left(x;q,\Sigma\right) & =\left(\pi^{\frac{n}{2}}\left|\Sigma\right|^{\nicefrac{1}{2}}\frac{\Gamma\left(\frac{1}{q-1}-\frac{n}{2}\right)}{\Gamma\left(\frac{1}{q-1}\right)}\cdot\left(\frac{2}{q-1}-n\right)^{\frac{n}{2}+\frac{1}{1-q}}\right)^{-1}\left(\left(\frac{2}{q-1}-n\right)+\left\langle x,\Sigma^{-1}x\right\rangle \right)^{\frac{1}{1-q}}\\
 & =\left(\pi^{\frac{n}{2}}\left|\Sigma\right|^{\nicefrac{1}{2}}\frac{\Gamma\left(\frac{1}{q-1}-\frac{n}{2}\right)}{\Gamma\left(\frac{1}{q-1}\right)}\cdot\left(\frac{2}{q-1}-n\right)^{\frac{n}{2}}\right)^{-1}\left(1+\left(\frac{2}{q-1}-n\right)^{-1}\left\langle x,\Sigma^{-1}x\right\rangle \right)^{\frac{1}{1-q}}\\
 & =\frac{1}{\left|\pi\Sigma\right|^{\nicefrac{1}{2}}}\cdot\frac{\Gamma\left(\frac{1}{q-1}\right)}{\Gamma\left(\frac{1}{q-1}-\frac{n}{2}\right)}\cdot\left(\frac{2}{q-1}-n\right)^{-\frac{n}{2}}\cdot\left(1+\left(\frac{2}{q-1}-n\right)^{-1}\cdot\left\langle x,\Sigma^{-1}x\right\rangle \right)^{\frac{1}{1-q}}.
\end{align*}
\end{proof}
When $q\geq1+\frac{2}{n}$, the solution to problem \eqref{eq:generalized-entropy-obj} does not exist, due to property \ref{enu:maximizing-format} and discussions in the proof. In Lemma \ref{lem:density-normalization}, the presented density \eqref{eq:$q$Gaussian-heavy-tail-density}
outlines a formula for multivariate bell-curve distributions dependent
on the value of $q$. As $q$ shifts from values approaching $1$
from above to values approaching $1+\frac{2}{n}$ form below, the
resulting density transitions from Gaussian through a scaled version
of the multivariate $t-$distribution to Cauchy and beyond. This density
explicitly details all parameters, enabling the application of the
maximum likelihood principle and facilitating the use of maximum likelihood
estimation in modeling correlated data performed in Section \ref{sec:optimizing-penalized-mle}. 

In the context of \eqref{eq:$q$Gaussian-heavy-tail-density}, the \emph{characterization
matrix} \cite{Costa2003}, denoted by $\Sigma$, can undergo modifications
to include the degree of freedom parameter $m\coloneqq\frac{2}{q-1}-n$
\cite{Vignat2005}; specifically,
\begin{align}
p\left(x;q,\Lambda\right) & =\frac{1}{\left|\pi\Lambda\right|^{\nicefrac{1}{2}}}\cdot\frac{\Gamma\left(\frac{m}{2}+\frac{n}{2}\right)}{\Gamma\left(\frac{m}{2}\right)}\cdot\left(1+\left\langle x,\Lambda^{-1}x\right\rangle \right)^{\frac{1}{1-q}},\label{eq:$q$Gaussian-simple-form}\\
\text{where }\Lambda & \coloneqq m\Sigma.\nonumber \\
\text{and }m & \coloneqq\frac{2}{q-1}-n\nonumber 
\end{align}
Below are a few useful remarks related to the $q$Gaussian distribution
and other bell-curve distributions. 
\begin{rem}
\label{rem:incorporate-location} To incorporate the location parameter
$\mu$, \eqref{eq:$q$Gaussian-heavy-tail-density} and \eqref{eq:$q$Gaussian-simple-form}
become 
\begin{align*}
p\left(x;\mu,q,\Sigma\right) & =\frac{1}{\left|\pi\Sigma\right|^{\nicefrac{1}{2}}}\cdot\frac{\Gamma\left(\frac{1}{q-1}\right)}{\Gamma\left(\frac{1}{q-1}-\frac{n}{2}\right)}\cdot\left(\frac{2}{q-1}-n\right)^{-\frac{n}{2}}\\
 & \qquad\cdot\left(1+\left(\frac{2}{q-1}-n\right)^{-1}\cdot\left\langle x-\mu,\Sigma^{-1}\left(x-\mu\right)\right\rangle \right)^{\frac{1}{1-q}};\\
p\left(x;\mu,q,\Lambda\right) & =\frac{1}{\left|\pi\Lambda\right|^{\nicefrac{1}{2}}}\cdot\frac{\Gamma\left(\frac{m}{2}+\frac{n}{2}\right)}{\Gamma\left(\frac{m}{2}\right)}\cdot\left(1+\left\langle x-\mu,\Lambda^{-1}\left(x-\mu\right)\right\rangle \right)^{\frac{1}{1-q}}.
\end{align*}
\end{rem}

\begin{rem}
For random vector $X\sim q\text{Gaussian}\left(q,\Sigma\right)$,
its $q-$covariance is 
\begin{equation}
\mathbb{E}_{q}\left[XX^{T}\right]=\left(\frac{1}{q-1}-\frac{n}{2}\right)^{\frac{1-q}{2}}\left|\pi\Sigma\right|^{\frac{1-q}{2}}\cdot\frac{\Gamma\left(\frac{q}{q-1}-\frac{n}{2}\right)/\left(\Gamma\left(\frac{1}{q-1}-\frac{n}{2}\right)\right)^{q}}{\Gamma\left(\frac{q}{q-1}\right)/\left(\Gamma\left(\frac{1}{q-1}\right)\right)^{q}}\cdot\Sigma.
\end{equation}
\end{rem}

\begin{proof}
We note that 
\begin{align*}
\int_{\mathbb{R}^{n}}p^{q}\left(x;q,\Sigma\right)dx & =\int_{\mathbb{R}^{n}}\left(\frac{1}{\left|\pi\Lambda\right|^{\nicefrac{1}{2}}}\cdot\frac{\Gamma\left(\frac{1}{q-1}\right)}{\Gamma\left(\frac{1}{q-1}-\frac{n}{2}\right)}\cdot\left(1+\left\langle x,\Lambda^{-1}x\right\rangle \right)^{\frac{1}{1-q}}\right)^{q}dx\\
 & =\left(\frac{1}{\left|\pi\Lambda\right|^{\nicefrac{1}{2}}}\cdot\frac{\Gamma\left(\frac{1}{q-1}\right)}{\Gamma\left(\frac{1}{q-1}-\frac{n}{2}\right)}\right)^{q}\cdot\int_{\mathbb{R}^{n}}\left(\left(1+\left\langle x,\Lambda^{-1}x\right\rangle \right)^{\frac{1}{1-q}}\right)^{q}dx\\
 & =\left(\frac{1}{\left|\pi\Lambda\right|^{\nicefrac{1}{2}}}\cdot\frac{\Gamma\left(\frac{1}{q-1}\right)}{\Gamma\left(\frac{1}{q-1}-\frac{n}{2}\right)}\right)^{q}\cdot\int_{\mathbb{R}^{n}}\left(1+\left\langle x,\Lambda^{-1}x\right\rangle \right)^{\frac{q}{1-q}}dx\\
 & =\left(\frac{1}{\left|\pi\Lambda\right|^{\nicefrac{1}{2}}}\cdot\frac{\Gamma\left(\frac{1}{q-1}\right)}{\Gamma\left(\frac{1}{q-1}-\frac{n}{2}\right)}\right)^{q}\left|\Lambda\right|^{\nicefrac{1}{2}}\cdot\int_{\mathbb{R}^{n}}\left(1+\left\langle x,x\right\rangle \right)^{\frac{q}{1-q}}dx\\
 & =2^{n}\left(\frac{1}{\left|\pi\Lambda\right|^{\nicefrac{1}{2}}}\cdot\frac{\Gamma\left(\frac{1}{q-1}\right)}{\Gamma\left(\frac{1}{q-1}-\frac{n}{2}\right)}\right)^{q}\left|\Lambda\right|^{\nicefrac{1}{2}}\cdot\int_{\mathbb{R}_{>0}^{n}}\left(1+\left\langle x,x\right\rangle \right)^{\frac{q}{1-q}}dx\\
 & =2^{n}\left(\frac{1}{\left|\pi\Lambda\right|^{\nicefrac{1}{2}}}\cdot\frac{\Gamma\left(\frac{1}{q-1}\right)}{\Gamma\left(\frac{1}{q-1}-\frac{n}{2}\right)}\right)^{q}\left|\Lambda\right|^{\nicefrac{1}{2}}\cdot\int_{0}^{\infty}r^{n-1}\left(1+r^{2}\right)^{\frac{q}{1-q}}\\
 & \qquad\cdot\left(\prod_{i=1}^{n-2}\int_{0}^{\frac{\pi}{2}}\sin^{n-1-i}(\theta_{i})d\theta_{i}\cdot\int_{0}^{\frac{\pi}{2}}1d\theta\right)dr\\
 & =\pi2^{n-1}\left(\frac{1}{\left|\pi\Lambda\right|^{\nicefrac{1}{2}}}\cdot\frac{\Gamma\left(\frac{1}{q-1}\right)}{\Gamma\left(\frac{1}{q-1}-\frac{n}{2}\right)}\right)^{q}\left|\Lambda\right|^{\nicefrac{1}{2}}\cdot\left(\prod_{i=1}^{n-2}\frac{1}{2}\frac{\Gamma\left(\frac{n-i}{2}\right)\sqrt{\pi}}{\Gamma\left(\frac{n-i+1}{2}\right)}\right)\\
 & \qquad\cdot\int_{0}^{\infty}r^{n-1}\left(1+r^{2}\right)^{\frac{q}{1-q}}dr\\
 & =2\pi^{\frac{n}{2}}\left(\frac{1}{\left|\pi\Lambda\right|^{\nicefrac{1}{2}}}\cdot\frac{\Gamma\left(\frac{1}{q-1}\right)}{\Gamma\left(\frac{1}{q-1}-\frac{n}{2}\right)}\right)^{q}\left|\Lambda\right|^{\nicefrac{1}{2}}\cdot\left(\Gamma\left(\frac{n}{2}\right)\right)^{-1}\cdot\int_{0}^{\infty}\left(r^{1-n}\left(1+r^{2}\right)^{\frac{q}{q-1}}\right)^{-1}dr\\
 & =2\pi^{\frac{n}{2}}\left(\frac{1}{\left|\pi\Lambda\right|^{\nicefrac{1}{2}}}\cdot\frac{\Gamma\left(\frac{1}{q-1}\right)}{\Gamma\left(\frac{1}{q-1}-\frac{n}{2}\right)}\right)^{q}\left|\Lambda\right|^{\nicefrac{1}{2}}\cdot\left(\Gamma\left(\frac{n}{2}\right)\right)^{-1}\cdot\frac{1}{2}B\left(\frac{q}{q-1}-\frac{n}{2},\frac{n}{2}\right)\\
 & =\pi^{\frac{n}{2}}\left(\frac{1}{\left|\pi\Lambda\right|^{\nicefrac{1}{2}}}\cdot\frac{\Gamma\left(\frac{1}{q-1}\right)}{\Gamma\left(\frac{1}{q-1}-\frac{n}{2}\right)}\right)^{q}\left|\Lambda\right|^{\nicefrac{1}{2}}\cdot\left(\Gamma\left(\frac{n}{2}\right)\right)^{-1}\cdot\frac{\Gamma\left(\frac{q}{q-1}-\frac{n}{2}\right)\Gamma\left(\frac{n}{2}\right)}{\Gamma\left(\frac{q}{q-1}\right)}\\
 & =\left(\frac{1}{\left|\pi\Lambda\right|^{\nicefrac{1}{2}}}\cdot\frac{\Gamma\left(\frac{1}{q-1}\right)}{\Gamma\left(\frac{1}{q-1}-\frac{n}{2}\right)}\right)^{q}\left|\pi\Lambda\right|^{\nicefrac{1}{2}}\cdot\frac{\Gamma\left(\frac{q}{q-1}-\frac{n}{2}\right)}{\Gamma\left(\frac{q}{q-1}\right)}\\
 & =\left|\pi\Lambda\right|^{\frac{1-q}{2}}\cdot\frac{\Gamma\left(\frac{q}{q-1}-\frac{n}{2}\right)/\left(\Gamma\left(\frac{1}{q-1}-\frac{n}{2}\right)\right)^{q}}{\Gamma\left(\frac{q}{q-1}\right)/\left(\Gamma\left(\frac{1}{q-1}\right)\right)^{q}}.
\end{align*}
From \eqref{eq:q-second-moment-constraint}, we then have the following
expression for the $q-$variance-covariance matrix 
\begin{align}
\mathbb{E}_{q}\left[XX^{T}\right] & =\int_{\mathbb{R}^{n}}xx^{T}\cdot p^{q}\left(x;q,\Sigma\right)dx\nonumber \\
 & =\int_{\mathbb{R}^{n}}p^{q}\left(x;q,\Sigma\right)dx\cdot\Sigma\nonumber \\
 & =\left|\pi\Lambda\right|^{\frac{1-q}{2}}\cdot\frac{\Gamma\left(\frac{q}{q-1}-\frac{n}{2}\right)/\left(\Gamma\left(\frac{1}{q-1}-\frac{n}{2}\right)\right)^{q}}{\Gamma\left(\frac{q}{q-1}\right)/\left(\Gamma\left(\frac{1}{q-1}\right)\right)^{q}}\cdot\Sigma\nonumber \\
 & =m^{\frac{1-q}{2}}\left|\pi\Sigma\right|^{\frac{1-q}{2}}\cdot\frac{\Gamma\left(\frac{q}{q-1}-\frac{n}{2}\right)/\left(\Gamma\left(\frac{1}{q-1}-\frac{n}{2}\right)\right)^{q}}{\Gamma\left(\frac{q}{q-1}\right)/\left(\Gamma\left(\frac{1}{q-1}\right)\right)^{q}}\cdot\Sigma\nonumber \\
 & =\left(\frac{1}{q-1}-\frac{n}{2}\right)^{\frac{1-q}{2}}\left|\pi\Sigma\right|^{\frac{1-q}{2}}\Sigma\cdot\frac{\Gamma\left(\frac{q}{q-1}-\frac{n}{2}\right)/\left(\Gamma\left(\frac{1}{q-1}-\frac{n}{2}\right)\right)^{q}}{\Gamma\left(\frac{q}{q-1}\right)/\left(\Gamma\left(\frac{1}{q-1}\right)\right)^{q}}.\label{eq:q-variance-covariance}
\end{align}
\end{proof}
\begin{rem}
Multivariate $t$ distributions with degree of freedom of $\frac{2}{q-1}-n$
and the scale matrix of $\Sigma$ are $q$Gaussian with shape parameter
$q$ and scale matrix $\Sigma$. 
\end{rem}

\begin{rem}
The multivariate Cauchy distributions \cite{Lee2014a} in $\mathbb{R}^{n}$
with scale matrix $\frac{1}{2}\Sigma$ are $q$Gaussian with shape
parameter $q=1+\frac{2}{n+1}$ and scale matrix $\Sigma$. 
\end{rem}

\begin{rem}
For a random vector $X\sim q\text{Gaussian}\left(q,\Sigma\right)$,
the variance-covariance matrix exists if and only if $q<1+\frac{2}{n+2}$;
following the same procedure to derive the variance-covariance matrix
for multivariate $t$ distribution yields that,\emph{ if existing,}
\begin{equation}
\mathbb{E}\left[XX^{T}\right]=\frac{m}{m-2}\Sigma,\label{eq:variance-covariance}
\end{equation}
where $m=\frac{2}{q-1}-n$. 
\end{rem}

The remarks above on $q$Gaussian distributions reveal their flexibility
in incorporating a location parameter, $\mu$, and adapting to multivariate
contexts through detailed formulas. Remarkably, these distributions
bridge with the class of multivariate bell curve distributions including
Gaussian, scaled $t$, and Cauchy distributions under certain conditions
on the shape parameter $q$. The elaboration of $q-$correlation and the $q-$variable-covariance matrix underscores the capability of
these distributions to model and understand the intricacies of correlated
data effectively. 

\section{\label{sec:optimization} Proximal Conjugate Gradient Algorithm }

As delineated in Section \ref{subsec:optimization-motivation}, our
optimization scenario is predominantly quadratic in nature; therefore,
the conjugate gradient approach has potential for fast convergence
and numerical stability. This insight forms the basis for our introduction
of a proximal conjugate gradient algorithm framework, tailored to
navigate the complexities introduced by the nonconvex penalized $q$Gaussian
likelihood function for sparse statistical learning. To lay the groundwork
for this discussion, we begin with a overview of relevant concepts
in variational and nonsmooth analysis, presented in Section \ref{subsec:variational-analysis}. The results presented in Section \ref{subsec:variational-analysis} can be found in recent textbooks on variational and nonsmooth analysis, such as \cite{Rockafellar2010, Clarke1990, Morduchovic2018, Mordukhovich2006, Mordukhovich2006a, Bauschke2011}. 

\subsection{\label{subsec:variational-analysis} A Review on Variational and
Nonsmooth Analysis }

Let $\mathcal{C}^{k,\alpha_{H}}$ with $k\in\mathbb{N}_{\geq0}$ and $\alpha_{H}\in\left[0,1\right]$ denote the function space such that $\forall F\in\mathcal{C}^{k,\alpha_{H}}$, $F$ is $k$th continuously differentiable, and $D^{k}F$ is globally H\"{o}lder continuous with exponent $\alpha_{H}$; clearly, when $\alpha_{H}=1$, $D^{k}F$ is globally Lipschitz continuous. In this subsection, we will state the results from variational and
nonsmooth analysis related to the following optimization problem:
\begin{equation}
\min\ _{x\in\mathbb{R}^{p+1}}f\left(x\right)\coloneqq g\left(x\right)+h\left(x\right),\label{eq:general-opt-problem}
\end{equation}
where $f\in\mathcal{C}^{0,0}\left(\mathbb{R}^{p+1},\mathbb{R}\right)$
is a locally-Lipschitz proper function, $g\in\mathcal{C}^{1,1}\left(\mathbb{R}^{p+1},\mathbb{R}\right)$
is globally $L_{\nabla g}-$smooth and possibly nonconvex, and $h\in\mathcal{C}^{0,0}\left(\mathbb{R}^{p+1},\mathbb{R}\right)$
is a convex locally-Lipschitz function, possibly nonsmooth.   The globally Lipschitz property
of $\nabla g$ can be alternatively addressed by carrying out the optimization
over a compact set. In such scenarios, given that $\nabla g$ is locally
Lipschitz, it inherently becomes globally Lipschitz when restricted
to a compact set. 

Results from convex analysis suggest that $g,h$ are Clarke regular;
thus, $f$ is Clarke regular. The Clarke's directional derivative,
defined by 
\begin{align*}
f^{\circ}\left(x;d\right) & \coloneqq\lim_{y\rightarrow x}\sup_{t\searrow0}\frac{f\left(y+td\right)-f\left(y\right)}{t}\\
 & =\inf_{\delta>0}\sup_{\left\Vert y-x\right\Vert \leq\delta,0<t<\delta}\frac{f\left(y+td\right)-f\left(y\right)}{t},
\end{align*}
exists for all $x\in\mathbb{R}^{p+1}$ since $f$ is Clarke regular.
The Clarke subdifferential, denoted by $\partial_{\circ}$, is a set-valued
mapping defined by 
\begin{equation}
\partial_{\circ}f\left(x\right)\coloneqq\left\{ \phi\in\mathbb{R}^{p+1}\vert\forall d\in\mathbb{R}^{p+1},\ \left\langle \phi,d\right\rangle \leq f^{\circ}\left(x;d\right)\right\} .
\end{equation}
Since $f$ is a locally Lipschitz function, $\forall x\in\mathbb{R}^{p+1},\ \partial_{\circ}f\left(x\right)\neq\emptyset$.
Fundamental convex analysis results show that $\forall x\in\mathbb{R}^{p+1},\ \partial_{\circ}f\left(x\right)$
is compact, convex, and upper-semicontinuous. $\forall x,d\in\mathbb{R}^{p+1}$,
and we also have 
\begin{equation}
f^{\circ}\left(x;d\right)=\max\ _{u\in\partial_{\circ}f\left(x\right)}\left\langle u,\frac{d}{\left\Vert d\right\Vert }\right\rangle .\label{eq:max-inner-product-def}
\end{equation}
Furthermore, \eqref{eq:max-inner-product-def} is upper-semicontinuous
with respect to $x$. Simple convex geometry results conclude that 
\begin{equation}
    \left\{ \left(v,-1\right)\vert v\in\partial_{\circ}f\left(x\right)\right\} =N_{\text{epi }f}\left(x,f\left(x\right)\right),
\end{equation}
where $N_{\text{epi }f}\left(x,f\left(x\right)\right)$ denotes the \emph{normal cone} to $\text{epi }f$ at the point $\left(x,f\left(x\right)\right)$. 

Since $g$ is smooth, $\partial_{\circ}g\left(x\right)=\left\{ \nabla g\left(x\right)\right\} $
is a singleton. Then $\partial_{\circ}f\left(x\right)=\partial_{\circ}g\left(x\right)+\partial_{\circ}h\left(x\right)$,
and 
\begin{align}
f^{\circ}\left(x;d\right) & =\max\ _{u\in\partial_{\circ}f\left(x\right)}\left\langle u,\frac{d}{\left\Vert d\right\Vert }\right\rangle \nonumber \\
 & =\max\ _{u\in\left(\nabla g\left(x\right)+\partial_{\circ}h\left(x\right)\right)}\left\langle u,\frac{d}{\left\Vert d\right\Vert }\right\rangle \nonumber \\
 & =\left\langle \nabla g\left(x\right),\frac{d}{\left\Vert d\right\Vert }\right\rangle +\max\ _{v\in\partial_{\circ}h\left(x\right)}\left\langle v,\frac{d}{\left\Vert d\right\Vert }\right\rangle \label{eq:clarke-subdifferential-decomposition}\\
 & =g^{\circ}\left(x;d\right)+h^{\circ}\left(x;d\right) \notag.
\end{align}

Let 
\begin{equation}
M_{\rho}t\left(x\right)\coloneqq\left(t\square\left(\frac{1}{2\rho}\left\Vert \cdot\right\Vert ^{2}\right)\right)\left(x\right)=\inf_{y\in\mathbb{R}^{p+1}}t\left(y\right)+\frac{1}{2\rho}\left\Vert y-x\right\Vert ^{2} \label{eq:moreau-envelope-defn}
\end{equation}
denote the Moreau envelope operator parameterized by $\rho\in\mathbb{R}_{>0}$
applied on an arbitrary proper, lower semi-continuous, locally Lipschitz function $t\in\mathcal{C}^{0,0}\left(\mathbb{R}^{p+1},\mathbb{R}\right)$,
where ``$\square$'' denotes the infimal convolution operator. We
have that the Moreau envelope is a smoothing operator, specifically,
\begin{equation}
\text{epi }t+\text{epi }\frac{1}{2\rho}\left\Vert \cdot\right\Vert ^{2}\subseteq \text{epi }M_{\rho}t,\label{eq:epi-decomposition}
\end{equation}
where ``$\text{epi}$'' denotes the epigraph. Clearly, $M_{\rho}t\left(x\right)\leq t\left(x\right)$, since $\left(0,0\right)\in\text{epi }\frac{1}{2\rho}\left\Vert \cdot\right\Vert ^{2}$ implies that $\text{epi }t=\text{epi }t+\left(0,0\right)\subseteq\text{epi }t+\text{epi }\frac{1}{2\rho}\left\Vert \cdot\right\Vert ^{2}\subseteq\text{epi }M_{\rho}t$. When $t$ is convex, \eqref{eq:epi-decomposition} takes the equal sign; i.e., the infimal convolution becomes the exact infimal convolution. 

Consider the affine function 
\begin{equation}
A\left(x\right)\coloneqq\left\langle a,x\right\rangle +b,
\end{equation}
simple algebra shows that the Moreau envelope applied on $A$ is 
\begin{equation}
M_{\rho}A\left(x\right)=\left\langle a,x\right\rangle +b+\frac{\rho}{2}\left\Vert a\right\Vert ^{2}=A\left(x\right)+\frac{\rho}{2}\left\Vert a\right\Vert ^{2}
\end{equation}
for some $a,b\in\mathbb{R}^{p+1}$. Moreover, the following affine
addition property is often used in proximal algorithms, mainly due
to the fact that the epigraph of an affine function is a half-space that
the Moreau envelope applied on: 
\begin{align}
M_{\rho}\left(t+A\right)\left(x\right) & =M_{\rho}t\left(x-\rho a\right)+\left\langle a,x\right\rangle +b-\frac{\rho}{2}\left\Vert a\right\Vert ^{2}\label{eq:Moreau-envelope-affine-addition}
\end{align}
Let 
\begin{equation}
\text{prox}_{\rho t}\left(x\right)\coloneqq\arg M_{\rho}t\left(x\right)=\underset{y\in\mathbb{R}^{p+1}}{\arg\min\ }t\left(y\right)+\frac{1}{2\rho}\left\Vert y-x\right\Vert ^{2}
\end{equation}
denote the proximal operator, a set-valued mapping; we have 
\begin{equation}
\text{prox}_{\rho t}=\left(I+\rho\partial_{\circ}t\right)^{-1}\label{eq:proximal-resolvency}
\end{equation}
is the resolvent of the Clarke's subdifferential operator $\rho\partial_{\circ}t$. 

For nonsmooth problems, proximal methods are often used. Fundamental
convex analysis results show that: 
\begin{enumerate}
\item the Moreau envelope $M_{\rho}t\left(x\right)$ is twice differentiable;
thus, its gradient $\nabla M_{\rho}t\left(x\right)$ is well-defined. 
\item If $t$ is convex, $\text{prox}_{\rho t}\left(x\right)$ is a singleton.
\emph{For the sake of parsimony, with a slight abuse of notation,
we use $\text{prox}_{\rho t}$ to represent a function in this case.}
It follows that both $\text{prox}_{\rho t}$ and $\nabla M_{\rho}t$
are firmly non-expansive, and Moreau's decomposition theorem implies that 
\begin{equation}
\nabla M_{\rho}t\left(x\right)=\rho^{-1}\left(x-\text{prox}_{\rho t}\left(x\right)\right).
\end{equation}
\end{enumerate}
The results from variational and nonsmooth analysis in this subsection
have laid the foundation for proving the properties discussed
in Section \ref{subsec:PCG-framework}. 

\subsection{\label{subsec:PCG-framework} Proximal Conjugate Gradient Framework }

Proximal methods are powerful optimization techniques and are particularly
adept at handling problems characterized by sparsity, which usually
leads to an optimization problem that is nonsmooth \cite{Nikolova2000}.
Proximal algorithm tends to outperform other methods by far for nonsmooth
problems \cite{Yu2017,Li2016}. On another ground, Krylov subspace
methods represent a cornerstone of numerical analysis, providing a
powerful framework for solving large-scale optimization problems efficiently
\cite{Saad2003}. Krylov subspace methods exhibit a remarkable property
of convergence acceleration and vastly improved numerical stability,
making them indispensable tools in the numerical analyst's toolkit. 

Having reviewed the related results from variational and non-smooth analysis in Section \ref{subsec:variational-analysis}, we are
ready to introduce our main optimization framework to combine proximal
methods and conjugate gradient together. The essence of proximal algorithms
lies upon the Moreau envelope's smoothing on the objective function. Indeed,
proximal methods minimize $M_{\rho}f$ instead of $f$, thus avoiding nonsmoothness since $M_{\rho}f$ is a smooth function. In this view,
proximal algorithms are, in fact, minimizing the Moreau envelope of
the objective function. Thus, a wide class of numerical optimization
algorithms can easily have their proximal version. Among those, conjugate
gradients, a type of Krylov subspace method, are the state-of-the-art
methods in smooth optimization due to their computational and memory
efficiency, scalability, and numerical stability. 

Prior to introducing our proximal conjugate gradient update framework,
we will first show the equivalency of the optimization problem to
minimize \eqref{eq:general-opt-problem} and its the Moreau envelope.
In nonconvex optimization, the main task for numerical optimization
is to find a Clarke stationary point of the objective function, for
which we show in Theorem \ref{lem:equivalent-clarke-stationary-point}
that the set of Clarke stationary point of $f$ is identical to that
of $M_{\rho}f$ for $\rho\in\left(0,L_{\nabla g}^{-1}\right)$. 
\begin{lem}
\label{lem:equivalent-clarke-stationary-point} $\forall\bar{x}\in\mathbb{R}^{p+1},\rho\in\left(0,L_{\nabla g}^{-1}\right)$,
\begin{equation}
0\in\partial_{\circ}f\left(\bar{x}\right)\Leftrightarrow\nabla M_{\rho}f\left(\bar{x}\right)=0,\label{eq:equiv-equilibrium}
\end{equation}
\end{lem}

\begin{proof}
Consider arbitrary $x\in\mathbb{R}^{p+1},\ \rho\in\left(0,L_{\nabla g}^{-1}\right)$.
As discussed previously, the gradient of the Moreau envelope $\nabla M_{\rho}f\left(x\right)=\rho^{-1}\left(x-\text{prox}_{\rho f}\left(x\right)\right)$
implies that 
\begin{equation}
\text{prox}_{\rho f}\left(x\right)=x-\rho\nabla M_{\rho}f\left(x\right),
\end{equation}
which implies the following first-order (necessary) optimality condition
for Clarke's stationary point: 
\begin{equation}
0\in\rho^{-1}\left(x-\rho\nabla M_{\rho}f\left(x\right)-x\right)+\partial_{\circ}f\left(x-\rho\nabla M_{\rho}f\left(x\right)\right).
\end{equation}
The relation above is simplified to 
\begin{equation}
\nabla M_{\rho}f\left(x\right)\in\partial_{\circ}f\left(x-\rho\nabla M_{\rho}f\left(x\right)\right)=\nabla g\left(x-\rho\nabla M_{\rho}f\left(x\right)\right)+\partial_{\circ}h\left(x-\rho\nabla M_{\rho}f\left(x\right)\right).\label{eq:moreau-gradient-subdifferential}
\end{equation}
Consider arbitrary $\bar{x}\in\mathbb{R}^{p+1},\rho\in\left(0,L_{\nabla g}^{-1}\right)$. 

``$\Rightarrow$'' of \eqref{eq:equiv-equilibrium}: 

Let $0\in\partial_{\circ}f\left(\bar{x}\right)=\nabla g\left(\bar{x}\right)+\partial_{\circ}h\left(\bar{x}\right)$;
i.e., $\bar{x}$ is a Clarke stationary point of $f$. Then $-\nabla g\left(\bar{x}\right)\in\partial_{\circ}h\left(\bar{x}\right)$.
Since $h$ is convex, \eqref{eq:moreau-gradient-subdifferential}
implies that 
\begin{equation}
\left\langle -\nabla g\left(\bar{x}\right)-\left(\nabla M_{\rho}f\left(\bar{x}\right)-\nabla g\left(\bar{x}-\rho\nabla M_{\rho}f\left(\bar{x}\right)\right)\right),\rho\nabla M_{\rho}f\left(\bar{x}\right)\right\rangle \geq0.
\end{equation}
Simplification gives 
\begin{equation}
\left\langle \nabla g\left(\bar{x}-\rho\nabla M_{\rho}f\left(\bar{x}\right)\right)-\nabla g\left(\bar{x}\right),\nabla M_{\rho}f\left(\bar{x}\right)\right\rangle \geq\left\Vert \nabla M_{\rho}f\left(\bar{x}\right)\right\Vert ^{2}.\label{eq:equiv-inequ-1}
\end{equation}
By Cauchy-Schwartz inequality, 
\begin{equation}
\left\langle \nabla g\left(\bar{x}-\rho\nabla M_{\rho}f\left(\bar{x}\right)\right)-\nabla g\left(\bar{x}\right),\nabla M_{\rho}f\left(\bar{x}\right)\right\rangle \leq L_{\nabla g}\cdot\rho\left\Vert \nabla M_{\rho}f\left(\bar{x}\right)\right\Vert ^{2}.\label{eq:equiv-inequ-2}
\end{equation}
Since $\rho<L_{\nabla g}^{-1}$, \eqref{eq:equiv-inequ-1} and \eqref{eq:equiv-inequ-2}
imply that 
\begin{equation}
\left\Vert \nabla M_{\rho}f\left(\bar{x}\right)\right\Vert ^{2}\leq\left\langle \nabla g\left(\bar{x}-\rho\nabla M_{\rho}f\left(\bar{x}\right)\right)-\nabla g\left(\bar{x}\right),\nabla M_{\rho}f\left(\bar{x}\right)\right\rangle <\left\Vert \nabla M_{\rho}f\left(\bar{x}\right)\right\Vert ^{2},
\end{equation}
which implies that 
\begin{equation}
\nabla M_{\rho}f\left(\bar{x}\right)=0;
\end{equation}
i.e., $\bar{x}$ is the stationary point of $M_{\rho}f$, hence a
Clarke's stationary point. 

``$\Leftarrow$'' of \eqref{eq:equiv-equilibrium}: 

Let $\nabla f_{\rho}\left(\bar{x}\right)=0$; i.e. $\bar{x}$ is a
stationary point of $M_{\rho}f$. It follows directly from \eqref{eq:moreau-gradient-subdifferential}
that 
\begin{equation}
0=\nabla M_{\rho}f\left(\bar{x}\right)\in\partial_{\circ}f\left(\bar{x}-\rho\nabla M_{\rho}f\left(\bar{x}\right)\right)=\partial_{\circ}f\left(\bar{x}\right);
\end{equation}
i.e., $\bar{x}$ is a Clarke stationary point of $f$. 
\end{proof}
The vast majority of optimization algorithms for smooth objective functions
require Lipschitz continuity of the objective function. Thus, we are
to propose the following Lemma to show the Lipschitz continuity of
the gradient of the Moreau envelope of $f$. 
\begin{lem}
\label{lem:moreau-Lipschitz} $\forall\rho\in\left(0,L_{\nabla g}^{-1}\right)$,
$\exists L_{\nabla M_{\rho}f}\in\mathbb{R}_{>0}$ such that 
\begin{equation}
\forall x,y\in\mathbb{R}^{p+1},\ \left\Vert \nabla M_{\rho}f\left(x\right)-\nabla M_{\rho}f\left(y\right)\right\Vert \leq L_{\nabla M_{\rho}f}\left\Vert x-y\right\Vert .
\end{equation}
\end{lem}

\begin{proof}
Consider arbitrary $x,y\in\mathbb{R}^{p+1}$. From \eqref{eq:moreau-gradient-subdifferential},
since $h$ is convex, 
\begin{dmath}
\left\langle \nabla M_{\rho}f\left(x\right)-\nabla g\left(x-\rho\nabla M_{\rho}f\left(x\right)\right)-\left(\nabla M_{\rho}f\left(y\right)-\nabla g\left(y-\rho\nabla M_{\rho}f\left(y\right)\right)\right),x-\rho\nabla M_{\rho}f\left(x\right)-\left(y-\rho\nabla M_{\rho}f\left(y\right)\right)\right\rangle \geq0.
\end{dmath}
Simplification gives 
\begin{dmath}
\left\langle \nabla M_{\rho}f\left(x\right)-\nabla M_{\rho}f\left(y\right)-\left(\nabla g\left(x-\rho\nabla M_{\rho}f\left(x\right)\right)-\nabla g\left(y-\rho\nabla M_{\rho}f\left(y\right)\right)\right),x-y-\rho\left(\nabla M_{\rho}f\left(x\right)-\nabla M_{\rho}f\left(y\right)\right)\right\rangle \geq0.
\end{dmath}
Let $\delta_{\nabla M_{\rho}f}\coloneqq\nabla M_{\rho}f\left(x\right)-\nabla M_{\rho}f\left(y\right)$,
$\delta_{\nabla g}\coloneqq\nabla g\left(x-\rho\nabla f_{\rho}\left(x\right)\right)-\nabla g\left(y-\rho\nabla f_{\rho}\left(y\right)\right)$,
and $\delta_{x,y}\coloneqq x-y$, then 
\begin{align*}
0 & \leq\left\langle \delta_{\nabla M_{\rho}f}-\delta_{\nabla g},\delta_{x,y}-\rho\delta_{\nabla M_{\rho}f}\right\rangle \\
 & =-\rho\left\Vert \delta_{\nabla M_{\rho}f}\right\Vert ^{2}+\rho\left\langle \delta_{\nabla g},\delta_{\nabla M_{\rho}f}\right\rangle +\left\langle \delta_{\nabla M_{\rho}f},\delta_{x,y}\right\rangle -\left\langle \delta_{\nabla g},\delta_{x,y}\right\rangle \\
 & \leq-\rho\left\Vert \delta_{\nabla M_{\rho}f}\right\Vert ^{2}+\rho\left\Vert \delta_{\nabla g}\right\Vert \cdot\left\Vert \delta_{\nabla M_{\rho}f}\right\Vert +\left\Vert \delta_{\nabla M_{\rho}f}\right\Vert \cdot\left\Vert \delta_{x,y}\right\Vert +\left\Vert \delta_{\nabla g}\right\Vert \cdot\left\Vert \delta_{x,y}\right\Vert \\
 & \leq-\rho\left\Vert \delta_{\nabla M_{\rho}f}\right\Vert ^{2}+\rho L_{\nabla g}\left(\left\Vert \delta_{x,y}\right\Vert +\rho\left\Vert \delta_{\nabla M_{\rho}f}\right\Vert \right)\cdot\left\Vert \delta_{\nabla M_{\rho}f}\right\Vert \\
 & \qquad+\left\Vert \delta_{\nabla M_{\rho}f}\right\Vert \cdot\left\Vert \delta_{x,y}\right\Vert +L_{\nabla g}\left(\left\Vert \delta_{x,y}\right\Vert +\rho\left\Vert \delta_{\nabla M_{\rho}f}\right\Vert \right)\cdot\left\Vert \delta_{x,y}\right\Vert 
\end{align*}
Simplification of the above inequality gives 
\begin{equation}
\left\Vert \delta_{\nabla M_{\rho}f}\right\Vert \leq\frac{2L_{g}\rho+1+\sqrt{8L_{g}\rho+1}}{2\rho\left(1-L_{\nabla g}\rho\right)}\left\Vert \delta_{x,y}\right\Vert ;
\end{equation}
i.e., 
\begin{equation}
\left\Vert \nabla M_{\rho}f\left(x\right)-\nabla M_{\rho}f\left(y\right)\right\Vert \leq L_{\nabla M_{\rho}f}\left\Vert x-y\right\Vert ,
\end{equation}
where 
\begin{equation}
L_{\nabla M_{\rho}f}\coloneqq\frac{2L_{g}\rho+1+\sqrt{8L_{g}\rho+1}}{2\rho\left(1-L_{\nabla g}\rho\right)}>0.
\end{equation}
Following this idea, we introduce our proximal conjugate gradient
framework in Algorithm \ref{alg:proximal-general-moreau}. 
\end{proof}
\begin{algorithm}[H]
\begin{algorithmic}[1]   \State Input: A fixed value of $\rho\in\left(0,\rho^{-1}\right)$ \State Calculate the gradient of the Moreau envelope: $s^{\left( k\right)} \coloneqq \nabla M_{\rho}f\left(x^{\left( k\right)}\right)$   \State $d^{\left( k\right)} \coloneqq -s^{\left( k\right)}+\beta^{\left( k\right)}\cdot d^{\left(k-1\right)}$   \State Line search to find $\alpha^{\left( k\right)}$ for the update $x^{\left(k+1\right)}\coloneqq x^{\left( k\right)}+\alpha^{\left(k\right)} d^{\left(k\right)}$ \State Update $x^{\left(k+1\right)} \coloneqq x^{\left( k\right)}+\alpha^{\left(k\right)} d^{\left(k\right)}$ \end{algorithmic} 

\caption{Proximal Point Algorithm \label{alg:proximal-general-moreau} }
\end{algorithm}

In the above algorithm, $\beta^{\left(k\right)}$ is the conjugate parameter. The
significant meaning of Algorithm \ref{alg:proximal-general-moreau}
is that for any global convergent numerical method to find the equilibria of a globally Lipschitz flow, which generally include the global convergent first-order methods, Algorithm \ref{alg:proximal-general-moreau} can transform
such a method to a proximal counterpart. 

For some objective functions, the gradient of the Moreau envelope
can be calculated directly. However, calculation for the Moreau envelope's
gradient is not tractable for many objective functions whose smooth
component $g$ is of complicated form. Motivated by this, we further
consider the following the Moreau envelope of the objective function
with linearized $g$, such linearization step is frequently used in
proximal algorithms for statistical sparse learning problems (e.g.,
\cite{Nesterov2004,Ghadimi2013,Yang2024}). 

Consider the linearized surrogate of \eqref{eq:general-opt-problem}, the locally Lipschitz function 
$\tilde{f}\in\mathcal{C}^{0,0}\left(\mathbb{R}^{p+1},\mathbb{R}\right)$,
defined by 
\begin{align}
\tilde{f}\left(x;u\right) & \coloneqq\left\langle u,x\right\rangle +h\left(x\right)\label{eq:linearizing-nonconvex}\\
\text{prox}_{\rho\tilde{f}}\left(x;u\right) & =\underset{y\in\mathbb{R}^{p+1}}{\arg\min\ }\left\{ \left\langle u,y\right\rangle +\frac{1}{2\rho}\left\Vert y-x\right\Vert ^{2}+h\left(y\right)\right\} \label{eq:practical-proximal}\\
\nabla_{x}M_{\rho}\tilde{f}\left(x;u\right) & =\rho^{-1}\left(x-\text{prox}_{\rho\tilde{f}}\left(x;u\right)\right)\label{eq:practical-moreau-gradient}
\end{align}
$\text{prox}_{\rho\tilde{f}}\left(x;u\right)$ is the proximal operator
applied on $\tilde{f}$, and $\nabla_{x}M_{\rho}\tilde{f}\left(x;u\right)$
is the gradient of the Moreau envelope of $\tilde{f}$. The linearization
term $\left\langle u,x\right\rangle $ in \eqref{eq:linearizing-nonconvex}
depends on $u$. Recognize that $\tilde{f}\left(x;u\right)$ is linearizing
the nonconvex smooth component $g$ in \eqref{eq:opt-problem} when
$u=\nabla g\left(x\right)$. 

We establish several definitions for subsequent utilization. Define
the mapping $\tilde{g}_{\rho}=I-\rho\nabla g\in \mathcal{C}^{0,0}\left(\mathbb{R}^{p+1},\mathbb{R}^{p+1}\right)$
\emph{for some $\rho\in\left(0,L_{\nabla g}^{-1}\right)$}, the locally Lipschitz property of $\tilde{g}_{\rho}$ follows from $g\in \mathcal{C}^{1,1}$; i.e.,
$\tilde{g}_{\rho}\left(x\right)\coloneqq x-\rho\nabla g\left(x\right).$
The following Lemma identifies some fundamental property of $\tilde{g}_{\rho}$. 
\begin{lem}
\label{lem:bijectivity} $\tilde{g}_{\rho}$ is a bijective from $\mathbb{R}^{p+1}$
to $\mathbb{R}^{p+1}$, and $\tilde{g}_{\rho}^{-1}$ is globally Lipschitz
with constant $\left(1-\rho L_{\nabla g}\right)^{-1}$. 
\end{lem}

\begin{proof}
\emph{Injectivity proof: }

Consider arbitrary $x_{1},x_{2}\in\mathbb{R}^{p+1}$. Since $\rho\in\left(0,L_{\nabla g}^{-1}\right)$,
$x_{1}-\rho\nabla g\left(x_{1}\right)=x_{2}-\rho\nabla g\left(x_{2}\right)$
implies that 
\begin{equation}
\left\Vert x_{1}-x_{2}\right\Vert =\rho\left\Vert \nabla g\left(x_{1}\right)-\nabla g\left(x_{2}\right)\right\Vert \leq\rho L_{\nabla g}\left\Vert x_{1}-x_{2}\right\Vert <\left\Vert x_{1}-x_{2}\right\Vert ,
\end{equation}
hence $x_{1}=x_{2}$. This shows that $\tilde{g}_{\rho}$ is a injective
mapping. 

\emph{Surjectivity proof: }

Consider arbitrary $y_{1},y_{2}\in\mathbb{R}^{p+1}$. Consider arbitrary
$z\in\mathbb{R}^{p+1}$. Define mapping $\mathcal{T}\left(y\right)\coloneqq z+\rho\nabla g\left(y\right)$,
then 
\begin{align*}
\left\Vert \mathcal{T}\left(y_{1}\right)-\mathcal{T}\left(y_{2}\right)\right\Vert  & =\left\Vert z+\rho\nabla g\left(y_{1}\right)-\left(z+\rho\nabla g\left(y_{2}\right)\right)\right\Vert \\
 & =\rho\left\Vert \nabla g\left(y_{1}\right)-\nabla g\left(y_{2}\right)\right\Vert \\
 & \leq\rho L_{\nabla g}\left\Vert y_{1}-y_{2}\right\Vert \\
 & <\left\Vert y_{1}-y_{2}\right\Vert .
\end{align*}
Thus, $\mathcal{T}$ is a contraction mapping, since $\mathbb{R}^{p+1}$
equipped with Euclidean topology is a Banach space, by Banach fixed
point theorem, $\mathcal{T}$ has a fixed point; i.e., $\exists y\in\mathbb{R}^{p+1}$
such that $y=z+\rho\nabla g\left(y\right)$, or equivalently, $\tilde{g}_{\rho}\left(y\right)=y-\rho\nabla g\left(y\right)=z$.
Thus, $\mathbb{R}^{p+1}\subseteq\tilde{g}_{\rho}\left(\mathbb{R}^{p+1}\right)$. 

\emph{Globally Lipschitz constant derivation for inverse map: }

Since $\nabla g$ is globally $L_{\nabla g}-$Lipschitz, 
\begin{align}
\left\Vert \tilde{g}_{\rho}\left(y_{1}\right)-\tilde{g}_{\rho}\left(y_{2}\right)\right\Vert  & =\left\Vert y_{1}-\rho\nabla g\left(y_{1}\right)-\left(y_{2}-\rho\nabla g\left(y_{2}\right)\right)\right\Vert \nonumber \\
 & =\left\Vert y_{1}-y_{2}-\rho\left(\nabla g\left(y_{1}\right)-\nabla g\left(y_{2}\right)\right)\right\Vert \nonumber \\
 & \geq\left|\left\Vert y_{1}-y_{2}\right\Vert -\left\Vert \rho\left(\nabla g\left(y_{1}\right)-\nabla g\left(y_{2}\right)\right)\right\Vert \right|\nonumber \\
 & =\left|\left\Vert y_{1}-y_{2}\right\Vert -\rho\left\Vert \nabla g\left(y_{1}\right)-\nabla g\left(y_{2}\right)\right\Vert \right|\nonumber \\
 & =\left\Vert y_{1}-y_{2}\right\Vert -\rho\left\Vert \nabla g\left(y_{1}\right)-\nabla g\left(y_{2}\right)\right\Vert \label{eq:triangle-inequ}\\
 & \geq\left(1-\rho L_{\nabla g}\right)\left\Vert y_{1}-y_{2}\right\Vert \nonumber 
\end{align}
where \eqref{eq:triangle-inequ} is due to the fact that 
\begin{equation}
\rho\left\Vert \nabla g\left(y_{1}\right)-\nabla g\left(y_{2}\right)\right\Vert \leq\rho L_{\nabla g}\left\Vert y_{1}-y_{2}\right\Vert <\left\Vert y_{1}-y_{2}\right\Vert .
\end{equation}
Since $\tilde{g}_{\rho}$ is surjective, consider arbitrary $z_{1},z_{2}\in\mathbb{R}^{p+1}$
let $y_{1}\coloneqq\tilde{g}_{\rho}^{-1}\left(z_{1}\right)$ and $y_{2}\coloneqq\tilde{g}_{\rho}^{-1}\left(z_{2}\right)$,
then 
\begin{equation}
\left\Vert \tilde{g}_{\rho}^{-1}\left(z_{1}\right)-\tilde{g}_{\rho}^{-1}\left(z_{2}\right)\right\Vert \leq\left(1-\rho L_{\nabla g}\right)^{-1}\left\Vert z_{1}-z_{2}\right\Vert .
\end{equation}
\end{proof}
Define 
\begin{equation}
\mathcal{G}_{\rho\tilde{f}}\left(x\right)\coloneqq\nabla_{x}M_{\rho}\tilde{f}\left(x;u\right)\label{eq:target-gradient-flow}
\end{equation}
with $u=\nabla g\left(x\right)$; i.e., $\mathcal{G}_{\rho\tilde{f}}\left(x\right)$
is the gradient of the Moreau envelope of $\tilde{f}$. 

Similarly to Lemma \ref{lem:equivalent-clarke-stationary-point} and
\ref{lem:moreau-Lipschitz}, we are to prove that the set of Clarke's
stationary of \eqref{eq:opt-problem} is identical to the set $\left\{ \bar{x}\in\mathbb{R}^{p+1}\vert\mathcal{G}_{\rho\tilde{f}}\left(\bar{x}\right)=0\right\} $
in Lemma \ref{lem:equivalent-clarke-stationary-point-pratical},
and then we are to show that \eqref{eq:target-gradient-flow} is globally
Lipschitz in Lemma \ref{lem:moreau-Lipschitz-practical}. 
\begin{lem}
\label{lem:equivalent-clarke-stationary-point-pratical} $\forall\bar{x}\in\mathbb{R}^{p+1},\rho\in\mathbb{R}_{>0}$,
\begin{equation}
0\in\partial_{\circ}f\left(\bar{x}\right)\Leftrightarrow\mathcal{G}_{\rho\tilde{f}}\left(\bar{x}\right)=0.\label{eq:equiv-equilibrium-practical}
\end{equation}
\end{lem}

\begin{proof}
Consider arbitrary $x\in \mathbb{R}^{p+1}$. The $\tilde{f}$ is convex since it is a sum of convex function $h$
and a linear mapping of $x$, which is convex. 
\begin{align}
\mathcal{G}_{\rho\tilde{f}}\left(x\right) & =\rho^{-1}\left(x-\text{prox}_{\rho\tilde{f}}\left(x;\nabla g\left(x\right)\right)\right)\label{eq:convex-moreau-envelope-grad}\\
 & =\rho^{-1}\left(x-\text{prox}_{\rho h}\left(x-\rho\nabla g\left(x\right)\right)\right)\label{eq:moreau-envelope-grad-decomposition}\\
 & =\nabla g\left(x\right)+\rho^{-1}\left(x-\rho\nabla g\left(x\right)-\text{prox}_{\rho h}\left(x-\rho\nabla g\left(x\right)\right)\right)\nonumber \\
 & =\nabla g\left(x\right)+\left(\nabla M_{\rho}h\right)\circ\tilde{g}_{\rho}\left(x\right)\label{eq:moreau-envelope-grad-decomposition-final-form}
\end{align}
\eqref{eq:moreau-envelope-grad-decomposition} is due to the affine
addition property of proximal mapping. From \eqref{eq:linearizing-nonconvex}
and \eqref{eq:convex-moreau-envelope-grad}, 
\begin{align}
\mathcal{G}_{\rho\tilde{f}}\left(x\right) & =\rho^{-1}\left(x-\text{prox}_{\rho\tilde{f}}\left(x,\nabla g\left(x\right)\right)\right)\nonumber \\
\implies\text{prox}_{\rho\tilde{f}}\left(x,\nabla g\left(x\right)\right) & =x-\rho\cdot\mathcal{G}_{\rho\tilde{f}}\left(x\right)\nonumber \\
\implies0 & \in\rho^{-1}\left(x-\rho\cdot\mathcal{G}_{\rho\tilde{f}}\left(x\right)-x\right)+\partial_{\circ}\tilde{f}\left(x-\rho\cdot\mathcal{G}_{\rho\tilde{f}}\left(x\right)\right)\nonumber \\
\implies0 & \in-\mathcal{G}_{\rho\tilde{f}}\left(x\right)+\nabla g\left(x\right)+\partial_{\circ}h\left(x-\rho\cdot\mathcal{G}_{\rho\tilde{f}}\left(x\right)\right)\nonumber \\
\implies\mathcal{G}_{\rho\tilde{f}}\left(x\right) & \in\nabla g\left(x\right)+\partial_{\circ}h\left(x-\rho\cdot\mathcal{G}_{\rho\tilde{f}}\left(x\right)\right)\label{eq:practical-subdifferential}\\
\implies\mathcal{G}_{\rho\tilde{f}}\left(x\right)-\nabla g\left(x\right) & \in\partial_{\circ}h\left(x-\rho\cdot\mathcal{G}_{\rho\tilde{f}}\left(x\right)\right).\nonumber 
\end{align}
Thus, since $h$ is convex, $\forall v\in\partial_{\circ}h\left(x\right)$,
\begin{align}
\left\langle \mathcal{G}_{\rho\tilde{f}}\left(x\right)-\nabla g\left(x\right)-v,x-\rho\cdot\mathcal{G}_{\rho\tilde{f}}\left(x\right)-x\right\rangle  & \geq0\nonumber \\
\implies\left\langle \mathcal{G}_{\rho\tilde{f}}\left(x\right)-\nabla g\left(x\right)-v,\mathcal{G}_{\rho\tilde{f}}\left(x\right)\right\rangle  & \leq0\nonumber \\
\implies\left\Vert \mathcal{G}_{\rho\tilde{f}}\left(x\right)\right\Vert ^{2} & \leq\left\langle \nabla g\left(x\right)+v,\mathcal{G}_{\rho\tilde{f}}\left(x\right)\right\rangle \label{eq:moreau-gradient-CS}\\
 & \leq\left\Vert \nabla g\left(x\right)+v\right\Vert \cdot\left\Vert \mathcal{G}_{\rho\tilde{f}}\left(x\right)\right\Vert \nonumber \\
\implies\left\Vert \mathcal{G}_{\rho\tilde{f}}\left(x\right)\right\Vert  & \leq\left\Vert \nabla g\left(x\right)+v\right\Vert ,\label{eq:moreau-gradient-minimal}
\end{align}
provided that $\left\Vert \mathcal{G}_{\rho\tilde{f}}\left(x\right)\right\Vert \neq0$.
Basic results on the Moreau envelope shows that $\mathcal{G}_{\rho\tilde{f}}\left(x\right)=0$
implies that $x$ is a Clarke stationary point of $\tilde{f}\left(x\right)$. 

Now we are proceed to prove \eqref{eq:equiv-equilibrium-practical}: 

``$\Rightarrow$'':

Consider arbitrary $\bar{x}\in\mathbb{R}^{p+1}$ and $\rho\in\mathbb{R}_{>0}$.
Let $0\in\partial_{\circ}f\left(\bar{x}\right)=\nabla g\left(\bar{x}\right)+\partial_{\circ}h\left(\bar{x}\right)$;
i.e., $\bar{x}$ is a Clarke stationary point of $f$. Then $\exists v\in\partial_{\circ}h\left(\bar{x}\right)$
such that $\nabla g\left(\bar{x}\right)+v=0$. \eqref{eq:moreau-gradient-minimal}
implies that 
\begin{equation}
\left\Vert \mathcal{G}_{\rho\tilde{f}}\left(\bar{x}\right)\right\Vert \leq\left\Vert \nabla g\left(\bar{x}\right)+v\right\Vert =0.
\end{equation}
Thus, $\mathcal{G}_{\rho\tilde{f}}\left(\bar{x}\right)=0$. 

``$\Leftarrow$'': 

Consider arbitrary $\bar{x}\in\mathbb{R}^{p+1}$. Let $\mathcal{G}_{\rho\tilde{f}}\left(\bar{x}\right)=0$;
i.e., $\mathcal{G}_{\rho\tilde{f}}\left(\bar{x}\right)=0$ is stationary.
\eqref{eq:practical-subdifferential} implies that 
\begin{equation}
0=\mathcal{G}_{\rho\tilde{f}}\left(\bar{x}\right)\in\nabla g\left(\bar{x}\right)+\partial_{\circ}h\left(\bar{x}-\rho\cdot\mathcal{G}_{\rho\tilde{f}}\left(\bar{x}\right)\right)=\nabla g\left(\bar{x}\right)+\partial_{\circ}h\left(\bar{x}\right)=\partial_{\circ}f\left(\bar{x}\right).
\end{equation}
Thus, $\bar{x}$ is a Clarke stationary point of $f$. 
\end{proof}
\begin{lem}
\label{lem:moreau-Lipschitz-practical} $\forall\rho\in\mathbb{R}_{>0}$,
$\exists L_{\mathcal{G}_{\rho\tilde{f}}}\in\mathbb{R}_{>0}$ such
that 
\begin{equation}
\forall x,y\in\mathbb{R}^{p+1},\ \left\Vert \mathcal{G}_{\rho\tilde{f}}\left(x\right)-\mathcal{G}_{\rho\tilde{f}}\left(y\right)\right\Vert \leq L_{\mathcal{G}_{\rho\tilde{f}}}\left\Vert x-y\right\Vert .
\end{equation}
\end{lem}

\begin{proof}
Consider arbitrary $x,y\in\mathbb{R}^{p+1}$ and $\rho\in\mathbb{R}_{>0}$.
Let $u\coloneqq\nabla g\left(x\right)$ and $v\coloneqq\nabla g\left(y\right)$,
\begin{align}
\left\Vert \mathcal{G}_{\rho\tilde{f}}\left(x\right)-\mathcal{G}_{\rho\tilde{f}}\left(y\right)\right\Vert  & =\left\Vert \nabla_{x}M_{\rho}\tilde{f}\left(x;u\right)-\nabla_{y}M_{\rho}\tilde{f}\left(y;v\right)\right\Vert \nonumber \\
 & =\left\Vert \nabla_{x}M_{\rho}\tilde{f}\left(x;u\right)-\nabla_{x}M_{\rho}\tilde{f}\left(x;v\right)+\nabla_{x}M_{\rho}\tilde{f}\left(x;v\right)-\nabla_{y}M_{\rho}\tilde{f}\left(y;v\right)\right\Vert \nonumber \\
 & \leq\left\Vert \nabla_{x}M_{\rho}\tilde{f}\left(x;u\right)-\nabla_{x}M_{\rho}\tilde{f}\left(x;v\right)\right\Vert +\left\Vert \nabla_{x}M_{\rho}\tilde{f}\left(x;v\right)-\nabla_{y}M_{\rho}\tilde{f}\left(y;v\right)\right\Vert \nonumber \\
 & \leq\left\Vert u-v\right\Vert +\left\Vert \nabla_{x}M_{\rho}\tilde{f}\left(x;v\right)-\nabla_{y}M_{\rho}\tilde{f}\left(y;v\right)\right\Vert \label{eq:lemma-4-result}\\
 & \leq\left\Vert u-v\right\Vert + \rho^{-1}\left\Vert x-y\right\Vert \label{eq:convex-moreau-grad-nonexpansive}\\
 & \leq L_{\nabla g}\left\Vert x-y\right\Vert + \rho^{-1}\left\Vert x-y\right\Vert \nonumber \\
 & =\left(L_{\nabla g}+\rho^{-1}\right)\left\Vert x-y\right\Vert \nonumber 
\end{align}
\eqref{eq:lemma-4-result} is due to Lemma 4 in \cite{Ghadimi2013},
and \eqref{eq:convex-moreau-grad-nonexpansive} is due to $\tilde{f}\left(\cdot;v\right)$
is convex and the fact that the gradient of a convex function's Moreau
envelope is $\rho^{-1}-$Lipschitz. Therefore, let 
\begin{equation}
L_{\mathcal{G}_{\rho\tilde{f}}}\coloneqq L_{\nabla g}+\rho^{-1}
\end{equation}
and we have 
\begin{equation}
\left\Vert \mathcal{G}_{\rho\tilde{f}}\left(x\right)-\mathcal{G}_{\rho\tilde{f}}\left(y\right)\right\Vert \leq L_{\mathcal{G}_{\rho\tilde{f}}}\left\Vert x-y\right\Vert .
\end{equation}
\end{proof}

Furthermore, \eqref{eq:moreau-envelope-grad-decomposition-final-form}
suggests that 
\begin{equation}
\mathcal{G}_{\rho\tilde{f}}=\nabla g+\left(\nabla M_{\rho}h\right)\circ\tilde{g}_{\rho}=\nabla g+\left(\nabla M_{\rho}h\right)\circ\left(Id-\rho\nabla g\right).
\end{equation}
Hence, 
\begin{align}
Id-\rho\mathcal{G}_{\rho\tilde{f}} & =Id-\rho\nabla g-\rho\left(\nabla M_{\rho}h\right)\circ\left(Id-\rho\nabla g\right)\nonumber \\
 & =\tilde{g}_{\rho}-\rho\left(\nabla M_{\rho}h\right)\circ\tilde{g}_{\rho}\nonumber \\
 & =\left(Id-\rho\left(\nabla M_{\rho}h\right)\right)\circ\tilde{g}_{\rho}\label{eq:flow-composition}\\
 & =\tilde{g}_{\rho}^{-1}\circ\left(\tilde{g}_{\rho}\circ\left(Id-\rho\left(\nabla M_{\rho}h\right)\right)\right)\circ\tilde{g}_{\rho}\label{eq:topologically-conjugate}
\end{align}
shows that $\tilde{g}_{\rho}^{-1}\circ\left(\tilde{g}_{\rho}\circ\left(Id-\rho\left(\nabla M_{\rho}h\right)\right)\right)\circ\tilde{g}_{\rho}:\mathbb{R}^{p+1}\mapsto\mathbb{R}^{p+1}$
equals to $Id-\rho\mathcal{G}_{\rho\tilde{f}}:\mathbb{R}^{p+1}\mapsto\mathbb{R}^{p+1}$.
Since $\tilde{g}_{\rho}$ ibijectivee, and that $\tilde{g}_{\rho}$
and $\tilde{g}_{\rho}^{-1}$ are continuous due to the globally Lipschitz
property from Lemma \ref{lem:bijectivity}, $\tilde{g}_{\rho}$ is
a homeomorphism. Hence, $Id-\rho\mathcal{G}_{\rho\tilde{f}}$ and
$\tilde{g}_{\rho}\circ\left(Id-\rho\left(\nabla M_{\rho}h\right)\right)$
are topologically equivalent  mappings via the homeomorphism $\tilde{g}_{\rho}$.
Lemma \ref{lem:moreau-Lipschitz-practical} implies that $\mathcal{G}_{\rho\tilde{f}}$
is globally Lipschitz, which sufficiently implies by the Cauchy-Lipschitz
theorem that the differential equation 
\begin{equation}
\dot{x}\coloneqq\frac{dx}{dt}=\mathcal{G}_{\rho\tilde{f}}\left(x\right)
\end{equation}
has a unique solution for any given initial value condition. Thus,
$\mathcal{G}_{\rho\tilde{f}}$ generates a unique flow under a given
initial value condition. 

The operator equations presented above can be understood as demonstrating
how $\mathcal{G}_{\rho\tilde{f}}$ functions analogously to a gradient
operator. Specifically, $Id-\rho\mathcal{G}_{\rho\tilde{f}}$ represents
executing a descent operation in the $-\mathcal{G}_{\rho\tilde{f}}$
direction with a step size $\rho$. Similarly, $Id-\rho\left(\nabla M_{\rho}h\right)$
represents a single gradient descent step with $\rho$ as the step
size with objective function $M_{\rho}h$, the Moreau envelope of
$h$; while $\tilde{g}_{\rho}=Id-\rho\nabla g$ reflects a gradient
descent step with objective function $g$, again with $\rho$ as the
step size. Equation \eqref{eq:flow-composition} elucidates that a
descent in the $-\mathcal{G}_{\rho\tilde{f}}$ direction is identical
to first performing a one-step gradient descent on $g$, followed
by $M_{\rho}h$; or performing gradient descents in a converse order
yields a topologically equivalence via the homeomorphism $\tilde{g}_{\rho}$,
as shown in \eqref{eq:topologically-conjugate}. 

In short summary, the approach based on linearization of the smooth term and the Moreau envelope
enables us to build equivalence between identifying Clarke stationary
points of the original nonsmooth objective function \eqref{eq:opt-problem}
and finding equilibria of the (unique) flow generated by $\mathcal{G}_{\rho\tilde{f}}$,
as demonstrated in Lemma \ref{lem:equivalent-clarke-stationary-point-pratical}.
The task of finding equilibria within a globally Lipschitz continuous
flow, such as the $\mathcal{G}_{\rho\tilde{f}}$ flow, is well explored
within mathematics, particularly in the realms of dynamical systems
and numerical analysis (see, for example, \cite{Quarteroni2007,Atkinson1989,Lubich2006,Hubbard1995,Helmke1994}).
Cauchy-Lipschitz theorem establishes the uniqueness of solutions to
initial value problems for globally Lipschitz continuous flows; while the existence of equilibria is a direct result of Brouwer fixed-point theorem. Numerical
methods for dynamical systems, including methods for finding equilibria
of the flow, are largely based on this uniqueness result. This is
one reason that the vast majority of numerical methods in the context
of dynamical systems require the flow to be globally Lipschitz. It is
important to note that these numerical strategies, widely applied
across dynamical systems, do not hinge on the flow being derived from
a conservative field. As such, the process of formulating a potential
function for $\mathcal{G}_{\rho\tilde{f}}$ is not a prerequisite
for employing numerical techniques to determine its equilibria. This
perspective underscores the versatility of numerical methods in dynamical
systems in finding the equilibria of flows, regardless of the explicit
existence of a potential function, a stance corroborated by various
sources in the literature \cite{Quarteroni2007,Atkinson1989,Lubich2006,Hubbard1995,Helmke1994,Ross2019,Riahi2018}.
In this view, the construction of a potential function for $\mathcal{G}_{\rho\tilde{f}}$
is generally not necessary when deploying numerical analysis methods
to find its equilibria. 

In the context of nonlinear conjugate gradient algorithms for optimization,
achieving global convergence on nonconvex objective functions that
are globally Lipschitz-smooth implies that such methods can reliably
find equilibria within the corresponding flow dynamics \cite{Ross2019,Riahi2018}.
These algorithms typically incorporate a line search step, which may
use a surrogate objective function instead of the original. This surrogate
can be a constructed potential, Lyapunov, or energy function, offering
flexibility when finding the potential function for $\mathcal{G}_{\rho\tilde{f}}$
poses challenges \cite{Ross2019,Clarke2004,Sontag1998}. 

When it is feasible to construct a potential function whose gradient with respect
to $x$ is $\mathcal{G}_{\rho\tilde{f}}$, the associated
objective function and its gradient become more manageable, allowing
for direct global convergence arguments. If constructing a potential
function with respect to $x$ for the $\left(\nabla M_{\rho}h\right)\circ\tilde{g}_{\rho}\left(x\right)$
term in \eqref{eq:moreau-envelope-grad-decomposition-final-form}
or $\nabla g\circ\left(Id-\rho\left(\nabla M_{\rho}h\right)\right)$
in \eqref{eq:topologically-conjugate} is tractable, the objective
function with gradient being \eqref{eq:moreau-envelope-grad-decomposition-final-form}
or $\tilde{g}_{\rho}\circ\left(Id-\rho\left(\nabla M_{\rho}h\right)\right)$
can hence be easily constructed. Thus, arguments for global convergence for methods
based on the objective function and its gradient directly follow
to prove the global convergence of the numerical optimization algorithm when
applied to the constructed potential function for $\mathcal{G}_{\rho\tilde{f}}$.
We remark that $Id-\rho\mathcal{G}_{\rho\tilde{f}}$ and $\tilde{g}_{\rho}\circ\left(Id-\rho\left(\nabla M_{\rho}h\right)\right)$
generate two topologically equivalent flows via homeomorphism $\tilde{g}_{\rho}$;
thus, their equilibria can be transformed by $\tilde{g}_{\rho}$ and
share the same stability. In the context of numerical optimization, this implies
that a fixed point $\bar{x}$ for the mapping $Id-\rho\mathcal{G}_{\rho\tilde{f}}$
corresponds bijectively to a fixed point $\tilde{g}_{\rho}\left(\bar{x}\right)$
for $\tilde{g}_{\rho}\circ\left(Id-\rho\left(\nabla M_{\rho}h\right)\right)$.
Characterized by the first-order optimality condition in optimization
of smooth functions, or equivalently, the stationary condition in dynamical
system, 
\begin{equation}
\mathcal{G}_{\rho\tilde{f}}\left(\bar{x}\right)=\nabla g\left(\bar{x}\right)+\left(\nabla M_{\rho}h\right)\circ\tilde{g}_{\rho}\left(\bar{x}\right)=0\Leftrightarrow\nabla M_{\rho}h\left(\tilde{g}_{\rho}\left(\bar{x}\right)\right)+\nabla g\left(\tilde{g}_{\rho}\left(\bar{x}\right)-\rho\nabla M_{\rho}h\left(\tilde{g}_{\rho}\left(\bar{x}\right)\right)\right)=0.
\end{equation}
This approach is practical because the literature on first-order numerical
optimization techniques frequently includes proofs of global convergence
for methods that depend on the objective function and its gradient
(for example, see \cite{Fletcher1964,Polak1969,Hestenes1952,Dai1999,Hager2005}).
Alternatively, construction of a potential function for $\mathcal{G}_{\rho\tilde{f}}$ is often not necessary due to the fact that fixed-point methods finding equilibria for a flow mostly establish convergence properties based on Banach fixed point theorem. This theorem guarantees convergence through intrinsic flow properties, obviating the need for a potential function \cite{Burden2001, Atkinson1989, Agarwal2009}. Conventionally, the use of line search based on the objective function and its gradient has been applied in some numerical methods to ensure global convergence. However, with the rapid growth of research in high--dimensional statistical machine learning and large-scale optimization, evaluations of the objective function often proven to be inefficient. Consequently, recent years have seen the exploration of two main alternatives. For instance, two different types of approaches for global convergent nonlinear conjugate gradient methods have been proposed without the conventional objective function-based line search procedure. One type of approach ensures global convergence by utilizing a line search mechanism that depends only on the nonlinear equation that generates the flow \cite{Feng2017, Snyman1985, Snyman2004, Kafka2019}; that is, the gradient function for smooth optimization, or $\mathcal{G}_{\rho\tilde{f}}$ in our case. As an example, under the smoothness assumption, the first-order optimality condition for an exact line search often solves for $\alpha$ with the current value $x^{\left(k\right)}$ and the search direction $d^{\left(k\right)}$ from $\left\langle \mathcal{G}_{\rho\tilde{f}}\left(x^{\left(k\right)}+\alpha\cdot d^{\left(k\right)}\right),d^{\left(k\right)}\right\rangle =0$, an equation dependent only on $\mathcal{G}_{\rho\tilde{f}}$ but not any surrogate objective function. From a practical perspective, this one-dimensional root finding problem can be carried out efficiently using the Brent root finding algorithm \cite{Brent1971}. The other approach suggests achieving global convergence either without the need for line search \cite{Shi2005, Chen2018, Sun2001, Wu2011a, Wang2006, Zhou2009} or by meeting a condition related to the Zoutendijk condition to replace the Wolfe-Powell conditions of sufficient descent (Armijo) and curvature \cite{Neumaier2024}. Additionally, in scenarios where the fulfillment of a sufficient descent (Armijo) condition is imperative, the formulation of a surrogate objective function becomes essential. Considering \eqref{eq:moreau-envelope-grad-decomposition-final-form}, where a surrogate objective is required for the line search phase, it could be formulated as: 
\begin{align}
 & g\left(x\right)+\left(M_{\rho}h\right)\circ\tilde{g}_{\rho}\left(x\right)\nonumber \\
 & =g\left(x\right)+\left(M_{\rho}h\right)\circ\tilde{g}_{\rho}\left(x\right)+\text{constant}\nonumber \\
 & =g\left(x\right)+\left\langle \nabla g\left(x\right),\text{prox}_{\rho h}\left(x-\rho\nabla g\left(x\right)\right)-x\right\rangle +\frac{1}{2\rho}\left\Vert \text{prox}_{\rho h}\left(x-\rho\nabla g\left(x\right)\right)-x\right\Vert ^{2}\label{eq:evaluate-obj}\\
 & \qquad+h\left(\text{prox}_{\rho h}\left(x-\rho\nabla g\left(x\right)\right)\right)+\text{constant}\nonumber 
\end{align}
This formulation, denoted as \eqref{eq:evaluate-obj}, represents
a quadratic approximation of $g$ plus the nonsmooth term $h$, evaluated
at $\text{prox}_{\rho h}\left(x-\rho\nabla g\left(x\right)\right)$. This type of formulation has often been used for the line search step in previous studies \cite{Beck2009, Kanzow2020}. 
The addition of the term $-\left\langle \nabla g\left(x\right),x\right\rangle $
acts as a constant in \eqref{eq:linearizing-nonconvex}, analogous
to fixing the value of $u$ as $\nabla g\left(x\right)$ for linearization.
This constant term, $-\left\langle \nabla g\left(x\right),x\right\rangle $,
doesn't alter the gradient of the Moreau envelope \eqref{eq:practical-moreau-gradient}
or the proximal point \eqref{eq:practical-proximal}, serving to frame
the quadratic approximation of $g\left(\text{prox}_{\rho h}\left(x-\rho\nabla g\left(x\right)\right)\right)$. 

Evaluation of $\text{prox}_{\rho h}$ in \eqref{eq:evaluate-obj}
is tractable and efficient for many functions, such as the $\ell_{1}$
norm commonly encountered in sparse statistical learning can be efficiently
computed via the soft-thresholding function. Given that line search
rules such as the Wolfe-Powell or Armijo-Goldstein conditions require
only the difference in the value of the objective function at two points to decide
on the step size, the constant term in \eqref{eq:evaluate-obj} can
be disregarded. Subsequent global convergence arguments stem from
the fixed-point theory analysis of the numerical methods deployed to find the equilibria of the $\mathcal{G}_{\rho\tilde{f}}$
flow. Another possible surrogate objective function inspired by the
quadratic Lyapunov function for the $\mathcal{G}_{\rho\tilde{f}}$
flow could be $\frac{1}{2}\left\Vert \mathcal{G}_{\rho\tilde{f}}\right\Vert ^{2}$,
attains its minimal value $0$ exactly at the $\mathcal{G}_{\rho\tilde{f}}$
flow's equilibria. This quadratic approach simplifies evaluation,
but it may not offer insights into the potential function's
landscape, potentially limiting the numerical algorithm's acceleration
capabilities if such an algorithm uses the landscape information to ensure the sufficient descent (Armijo) condition. Therefore, formulating the surrogate objective function
preserving the landscape of the original objective function as
outlined in \eqref{eq:evaluate-obj} is preferable. 

Building on the above discussion, we introduce our practical proximal
conjugate gradient framework in Algorithm \ref{alg:proximal-practical-moreau}. 

\begin{algorithm}[H]
\begin{algorithmic}[1]   \State Input: A fixed value of $\rho\in\left(0,\rho^{-1}\right)$ \State Calculate the proximal value $p^{\left( k\right)} \coloneqq \text{prox}_{{\rho}^{-1}h}\left(x^{\left( k\right)}-{\rho}^{-1}\cdot \nabla g\left(x^{\left( k\right)}\right)\right)$   \State Calculate $\mathcal{G}_{\rho\tilde{f}}\left(x^{\left( k\right)}\right)$: $s^{\left( k\right)} \coloneqq {\rho}\left(x^{\left( k\right)}-p^{\left( k\right)}\right)$   \State $d^{\left( k\right)} \coloneqq -s^{\left( k\right)}+\beta^{\left( k\right)}\cdot d^{\left(k-1\right)}$   \State Line search to find $\alpha^{\left( k\right)}$ for the update $x^{\left(k+1\right)}\coloneqq x^{\left( k\right)}+\alpha^{\left(k\right)} d^{\left(k\right)}$, if needed.  \State Update $x^{\left(k+1\right)} \coloneqq x^{\left( k\right)}+\alpha^{\left(k\right)} d^{\left(k\right)}$ \end{algorithmic} 

\caption{Computationally Tractable Proximal Conjugate Gradient Update Scheme \label{alg:proximal-practical-moreau} }
\end{algorithm}

In Algorithm \ref{alg:proximal-practical-moreau}, $\beta^{\left(k\right)}$ functions
as the conjugate parameter. Unlike Algorithm \ref{alg:proximal-general-moreau},
Algorithm \ref{alg:proximal-practical-moreau} facilitates the update
process without the need to compute $\nabla M_{\rho}f\left(x^{\left(k\right)}\right)$.
This adaptation is significantly valuable in practical scenarios,
especially in statistical sparse learning challenges characterized
by a complicated smooth component $g$ alongside a simple nonsmooth
convex component $h$. In such cases, computing $\text{prox}_{\rho h}$
is markedly more tractable and efficient than $\text{prox}_{\rho f}$.
This approach is particularly beneficial for sparse statistical learning
issues, where sparsity is commonly induced by an $\ell_{1}$ penalty
term. 

\subsection{\label{subsec:Proximal-HZ} Proximal Hager-Zhang \cite{Hager2005}
Conjugate Gradient }

The nonlinear conjugate gradient method represents the pinnacle of
first-order techniques for addressing smooth optimization challenges.
Various versions of nonlinear conjugate gradient methods have been
introduced, including the Fletcher-Reeves (FR) method \cite{Fletcher1964},
the modified Polak-Ribiere-Polyak (PRP+) method \cite{Polak1969,Gilbert1992},
the Hestenes-Stiefel (HS) method \cite{Hestenes1952}, the Dai-Yuan
(DY) method \cite{Dai1999}, and the Hager-Zhang (HZ) method \cite{Hager2005}.
These versions have all demonstrated global convergence with nonconvex
globally Lipschitz-smooth objective functions. Among these, the Hager-Zhang
conjugate gradient method is notable for delivering the best numerical
performance on large-scale datasets, as indicated in previous research
\cite{Hager2006}. Building on this, having introduced our practical
proximal conjugate gradient update mechanism in Algorithm \ref{alg:proximal-practical-moreau},
we aim to extend this approach by adapting the smooth Hager-Zhang
nonlinear conjugate gradient method to its proximal version
in Algorithm \ref{alg:proximal-HZ-CG}. 

\begin{algorithm}
\begin{algorithmic}[1]  \State \textbf{Input:} Initial point $x^{\left( 0\right)}$; $g\in\mathcal{C}^{1,1}\left(\mathbb{R}^{p+1},\mathbb{R}\right)$; locally-Lipschitz, convex $h\in\mathcal{C}^{0,0}\left(\mathbb{R}^{p+1},\mathbb{R}\right)$; the smoothing parameter for the Moreau envelope $\rho\in\left(0,\rho^{-1}\right)$; $k \coloneqq 0$  \State \textbf{Output:} $p$  \State $k += 1$  \State Calculate the gradient for $g$: $g^{\left( 0\right)} \coloneqq \nabla g\left(x^{\left( 0\right)}\right)$  \State Calculate the proximal value $p^{\left( 0\right)} \coloneqq \text{prox}_{\rho,h}\left(x^{\left( 0\right)}-\rho\cdot g^{\left( 0\right)}\right)$  \State Calculate the gradient analog: $s^{\left( 0\right)} \coloneqq x^{\left( 0\right)}-p^{\left( 0\right)}$  \State $d^{\left( 0\right)} \coloneqq -s^{\left( 0\right)}$  \State Perform the line search with $d^{\left( 0\right)}$ with step size $\alpha^{\left( 0\right)}$ \State Update $x_1 \coloneqq x^{\left( 0\right)} + \alpha^{\left( 0\right)} d^{\left( 0\right)}$  \While{not converged}  \State $k += 1$ \State Calculate the gradient for $g$: $g^{\left( k\right)} \coloneqq \nabla g\left(x^{\left( k\right)}\right)$  \State Calculate the proximal value $p^{\left( k\right)} \coloneqq \text{prox}_{\rho,h}\left(x^{\left( k\right)}-\rho\cdot g^{\left( k\right)}\right)$  \State Calculate the gradient analog: $s^{\left( k\right)} \coloneqq x^{\left( k\right)}-p^{\left( k\right)}$  \State $d^{\left( k\right)} \coloneqq -s^{\left( k\right)}+\bar{\beta}^{\left( k\right)}\cdot d^{\left(k-1\right)}$  \State Perform the line search with $d^{\left( k\right)}$ with step size $\alpha^{\left( k\right)}$ based on Wolfe-Powell conditions \label{eqn:d-k} \State Update $x^{\left(k+1\right)} \coloneqq x^{\left( k\right)} + \alpha^{\left( k\right)} d^{\left( k\right)}$  \State Check for convergence  \EndWhile  \State \textbf{return} $p^{\left( k\right)}$  \end{algorithmic} 

\caption{Proximal Hager-Zhang \cite{Hager2005} Conjugate Gradient \label{alg:proximal-HZ-CG}}
\end{algorithm}

In Algorithm \ref{alg:proximal-HZ-CG}, Hager-Zhang's conjugate parameter
$\bar{\beta}^{\left(k\right)}$ is defined as \cite{Hager2005}: 
\begin{align*}
y^{\left(k\right)} & \coloneqq s^{\left(k+1\right)}-s^{\left(k\right)}\\
\beta^{\left(k\right)} & \coloneqq\frac{1}{\left\langle d^{\left(k\right)},y^{\left(k\right)}\right\rangle }\cdot\left\langle y^{\left(k\right)}-2\frac{\left\Vert y^{\left(k\right)}\right\Vert ^{2}}{\left\langle d^{\left(k\right)},y^{\left(k\right)}\right\rangle }d^{\left(k\right)},s^{\left(k+1\right)}\right\rangle \\
\eta^{\left(k\right)} & \coloneqq-\frac{1}{\left\Vert d^{\left(k\right)}\right\Vert \min\ \left\{ \eta,\left\Vert s^{\left(k\right)}\right\Vert \right\} }\\
\bar{\beta}^{\left(k\right)} & \coloneqq\max\ \left\{ \beta^{\left(k\right)},\eta^{\left(k\right)}\right\} 
\end{align*}

It was proven that if the line search step in Algorithm \ref{alg:proximal-HZ-CG}
satisfies Wolfe-Powell conditions and the gradient is globally Lipschitz,
Hager-Zhang conjugate gradient achieves global convergence finding a stationary point for
a smooth nonconvex objective function. In a dynamical system view, this corresponds to the global attraction property of the trajectory of the numerical algorithm to find equilibria for globally Lipschitz flows.
Lemma \ref{lem:moreau-Lipschitz-practical} implies that $\mathcal{G}_{\rho\tilde{f}}$,
or $s^{\left(k\right)}$ in Algorithm \ref{alg:proximal-HZ-CG}, are globally Lipschitz.
Thus, by Lemma \ref{lem:equivalent-clarke-stationary-point-pratical},
Algorithm \ref{alg:proximal-HZ-CG} yields the Clarke stationary point
of $f$. Based on the arguments in Section \ref{subsec:PCG-framework}, if the potential
function for $\mathcal{G}_{\rho\tilde{f}}$ is tractable to construct, the Wolfe-Powell line search in Algorithm \ref{alg:proximal-HZ-CG} can be carried out using the potential function of $\mathcal{G}_{\rho\tilde{f}}$ as the surrogate objective function; alternatively, an exact line search can be carried out by finding $\alpha$ that satisfies $\left\langle \mathcal{G}_{\rho\tilde{f}}\left(x^{\left(k\right)}+\alpha\cdot d^{\left(k\right)}\right),d^{\left(k\right)}\right\rangle =0$ --- such an exact line search can usually be carried out efficiently using Brent's method to find a root of a one-dimensional equation in $\mathbb{R}_{>0}$ \cite{Brent1971}. Furthermore, the descent property of $d^{\left(k\right)}$ was shown by \cite{Hager2005} independent of the line searches, which guarantees that $\left\langle \mathcal{G}_{\rho\tilde{f}}\left(x^{\left(k\right)}+\alpha\cdot d^{\left(k\right)}\right),d^{\left(k\right)}\right\rangle =0$ has a positive root. Moreover, another line search to ensure global convergence can be carried out by backtracking to find $\alpha^{\left(k\right)}$ satisfying 
\begin{equation}
    -\left\langle \mathcal{G}_{\rho\tilde{f}}\left(x^{\left(k\right)}+c_{1}\cdot\alpha^{\left(k\right)}d^{\left(k\right)}\right),d^{\left(k\right)}\right\rangle \geq c_{1}c_{2}\cdot\alpha^{\left(k\right)}\left\Vert d^{\left(k\right)}\right\Vert ^{2} \label{eq:feng-backtracking-line-search},
\end{equation}
where $c_{1},c_{2}\in\mathbb{R}_{>0}$ are constant to be chosen. When $\mathcal{G}_{\rho\tilde{f}}$ is pseudo-monotone in the sense of Karamardian \cite{Karamardian1976}, since the global Lipschitz property was established for $\mathcal{G}_{\rho\tilde{f}}$ in Lemma \ref{lem:moreau-Lipschitz-practical}, global convergence was proven for this backtracking line search method \cite{Feng2017}. We conclude this section with the observation that certain conjugate gradient methods obviate the need for line search procedures by determining the step size directly from $s^{(k)}$ and $d^{(k)}$, as exemplified in \cite{Chen2018}.

\section{\label{sec:optimizing-penalized-mle} Optimizing Algorithm and Prediction for Penalized $q$Gaussian Likelihood Problems }

\subsection{Problem Formulation }

Using the $q$Gaussian distribution to model the data will undoubtedly
enhance the robustness towards the underlying distributional assumption
and outliers. However, unlike the Gaussian distribution, two independent
$q$Gaussian random vectors are not jointly $q$Gaussian. Thus, we
take the following approach to model the data. Let $\mathbf{X}_{\text{train}}\in\mathbb{R}^{n_{\text{train}}\times\left(p+1\right)}$, $\ y_{\text{train}}\in\mathbb{R}^{n_{\text{train}}}$
denote the training design matrix and outcome, $\mathbf{X}_{\text{val}}\in\mathbb{R}^{n_{\text{val}}\times\left(p+1\right)},\ y_{\text{val}}\in\mathbb{R}^{n_{\text{val}}}$
denote the validation design matrix and outcome, and $\mathbf{X}_{\text{test}}\in\mathbb{R}^{n_{\text{test}}\times\left(p+1\right)},\ y_{\text{test}}\in\mathbb{R}^{n_{\text{test}}}$
denote the testing design matrix and outcome. Let 
\begin{equation}
\mathbf{X}\coloneqq\left[\mathbf{X}_{\text{train}}^{T},\mathbf{X}_{\text{val}}^{T},\mathbf{X}_{\text{test}}^{T}\right]^{T}\in\mathbb{R}^{n\times\left(p+1\right)}
\end{equation}
denote the design matrix for the entire dataset, and let 
\begin{equation}
y\coloneqq\left[y_{\text{train}}^{T},y_{\text{val}}^{T},y_{\text{test}}^{T}\right]^{T}\in\mathbb{R}^{n}
\end{equation}
denote the outcome for the entire dataset. Instead of assuming the $q$Gaussian
distribution for the training, validation and testing set separately,
we assume that
\begin{equation}
y\sim q\text{Gaussian}\left(q,\mathbf{X}\theta,\Sigma\right),
\end{equation}
where $\theta\in\mathbb{R}^{p+1}$ denotes the coefficients for regression,
and $\Sigma$ denotes the characteristic/scale matrix for the entire
data. Clearly, 
\begin{align*}
\mathbf{X}_{\text{train}} & =\left[I_{n_{\text{train}}\times n_{\text{train}}},0_{n_{\text{train}}\times n_{\text{val}}},0_{n_{\text{train}}\times n_{\text{test}}}\right]\mathbf{X}\\
y_{\text{train}} & =\left[I_{n_{\text{train}}\times n_{\text{train}}},0_{n_{\text{train}}\times n_{\text{val}}},0_{n_{\text{train}}\times n_{\text{test}}}\right]y
\end{align*}
implies that 
\begin{equation}
y_{\text{train}}\sim q\text{Gaussian}\left(q_{\text{train}},\mathbf{X}_{\text{train}}\theta,\Sigma_{\text{train}}\right)\label{eq:$q$Gaussian-distributed}
\end{equation}
where by the linear mapping closeness property \ref{enu:linear-mapping-invariance},
\begin{equation}
\Sigma_{\text{train}}=\left[I_{n_{\text{train}}\times n_{\text{train}}},0_{n_{\text{train}}\times n_{\text{val}}},0_{n_{\text{train}}\times n_{\text{test}}}\right]\Sigma\left[I_{n_{\text{train}}\times n_{\text{train}}},0_{n_{\text{train}}\times n_{\text{val}}},0_{n_{\text{train}}\times n_{\text{test}}}\right]^{T}
\end{equation}
is the $n_{\text{train}}\times n_{\text{train}}$ block diagonal matrix
of $\Sigma$ corresponding to the training data. By \eqref{enu:linear-mapping-invariance},
\begin{equation}
\frac{2}{1-q_{\text{train}}^{-1}}-n_{\text{train}}=\frac{2}{1-q^{-1}}-n,
\end{equation}
which implies that 
\begin{equation}
\frac{1}{q_{\text{train}}-1}-n_{\text{train}}=\frac{1}{q-1}-n.\label{eq:recover-q}
\end{equation}
\eqref{eq:recover-q} allows us to recover $q$ from the training
procedure. Above formulas for the training data and parameters can
trivially be applied to the validation and the testing data and parameters;
thus, validation and test can carried out easily from the model build
from the training data. 

For $q-$correlated data, often times, the $q-$correlation structure
is inferred or given prior to the model fitting; thus, we assume that
the $q-$correlation structure is given as $\Psi$ and we estimate the volatility / dispersion / scale parameter $\sigma^{2}>0$ such that
\begin{equation}
\Sigma=\sigma^{2}\Psi.
\end{equation}
Trivially, $\Psi_{\text{train}}$ is the block diagonal matrix of
$\Psi$ corresponding to the training data and $\Sigma_{\text{train}}=\sigma^{2}\Psi_{\text{train}}$. 

We are now ready to formulate our likelihood loss function. To utilize
$q$Gaussian distribution to model the $q-$correlated observations, we estimate the value of $q$ such that $q$ is allowed to
vary, and the model will thus be more robust towards a wide class of distributions.
Therefore, we choose to build the model using \eqref{eq:$q$Gaussian-heavy-tail-density},
since the dispersion matrix $\Lambda$ \eqref{eq:$q$Gaussian-simple-form}
depends on $q$. We formulate our maximization of our log-likelihood function
as the following from \eqref{eq:$q$Gaussian-distributed} and \eqref{eq:$q$Gaussian-heavy-tail-density}:
\begin{dgroup*}
  \begin{dmath*}
\underset{q_{\text{train}}\in(1,1+\frac{2}{n_{\text{train}}}),\theta\in\mathbb{R}^{p+1},\sigma^{2}\in\mathbb{R}_{>0}}{\arg\max\ }  \log(\frac{1}{\left|\sigma^{2}\Psi_{\text{train}}\right|^{\nicefrac{1}{2}}}\cdot\frac{\Gamma\left(\frac{1}{q_{\text{train}}-1}\right)}{\Gamma\left(\frac{1}{q_{\text{train}}-1}-\frac{n_{\text{train}}}{2}\right)}\cdot\left(\frac{2}{q_{\text{train}}-1}-n_{\text{train}}\right)^{-\frac{n_{\text{train}}}{2}}
  \cdot\left(1+\left(\frac{2}{q_{\text{train}}-1}-n_{\text{train}}\right)^{-1}\cdot\left\langle y_{\text{train}}-\mathbf{X}_{\text{train}}\theta,\left(\sigma^{2}\Psi_{\text{train}}\right)^{-1}\left(y_{\text{train}}-\mathbf{X}_{\text{train}}\theta\right)\right\rangle \right)^{\frac{1}{1-q_{\text{train}}}}).
\end{dmath*}
\end{dgroup*}
To address the high--dimensional data concerns, Oracle penalties are
incorporated to carry out variable selection. To penalize the log-likelihood
loss function to achieve variable selection, we formulate the following
problem: 
\begin{dgroup*}
  \begin{dmath*}
\underset{q_{\text{train}}\in(1,1+\frac{2}{n_{\text{train}}}),\theta\in\mathbb{R}^{p+1},\sigma^{2}\in\mathbb{R}_{>0}}{\arg\min\ }-\log(\frac{1}{\left|\sigma^{2}\Psi_{\text{train}}\right|^{\nicefrac{1}{2}}}\cdot\frac{\Gamma\left(\frac{1}{q_{\text{train}}-1}\right)}{\Gamma\left(\frac{1}{q_{\text{train}}-1}-\frac{n_{\text{train}}}{2}\right)}\cdot\left(\frac{2}{q_{\text{train}}-1}-n_{\text{train}}\right)^{-\frac{n_{\text{train}}}{2}}
 \qquad\cdot\left(1+\left(\frac{2}{q_{\text{train}}-1}-n_{\text{train}}\right)^{-1}\cdot\sigma^{-2}\cdot\left(\left\langle y_{\text{train}}-\mathbf{X}_{\text{train}}\theta,\Psi_{\text{train}}^{-1}\left(y_{\text{train}}-\mathbf{X}_{\text{train}}\theta\right)\right\rangle +2n_{\text{train}}\sum_{j=2}^{p+1}w\left(\theta_{j}\right)\right)\right)^{\frac{1}{1-q_{\text{train}}}})\nonumber 
 \end{dmath*}
\begin{dmath}
\Leftrightarrow  \underset{q_{\text{train}}\in(1,1+\frac{2}{n_{\text{train}}}),\theta\in\mathbb{R}^{p+1},\sigma^{2}\in\mathbb{R}_{>0}}{\arg\min\ }\frac{n}{2}\log\sigma^{2}-\log\Gamma\left(\frac{1}{q_{\text{train}}-1}\right)+\log\Gamma\left(\frac{1}{q_{\text{train}}-1}-\frac{n_{\text{train}}}{2}\right)+\frac{n_{\text{train}}}{2}\log\left(\frac{2}{q_{\text{train}}-1}-n_{\text{train}}\right)
 \quad+\frac{1}{q_{\text{train}}-1}\log\left(1+\left(\frac{2}{q_{\text{train}}-1}-n_{\text{train}}\right)^{-1}\cdot\sigma^{-2}\cdot\left(\left\langle y_{\text{train}}-\mathbf{X}_{\text{train}}\theta,\Psi_{\text{train}}^{-1}\left(y_{\text{train}}-\mathbf{X}_{\text{train}}\theta\right)\right\rangle +2n_{\text{train}}\sum_{j=2}^{p+1}w\left(\theta_{j}\right)\right)\right)\label{eq:opt-problem}
\end{dmath}
\end{dgroup*}
In the above formulated problem, $w$ is the Oracle penalty function,
and we are not to penalize the intercept term. The $2n_{\text{train}}$
multiplier is to ensure that the penalization effect is consistent with
the number of training observations. We choose to put the penalty
term together with the quadratic term without the variance scale parameter
$\sigma^{2}$ for two reasons: first, the optimization problem is
more tractable under such problem formulation; second, we do not wish
to let the value of $\sigma^{2}$ perturb the degree of penalization.
Comparing to penalized the log-likelihood directly, we choose to penalize
the quadratic component directly as it is more tractable. It was shown
that doing so will preserve Oracle properties \cite{Nikolova2000}
of penalized estimators. 

For the optimization procedure, we will proceed in a blockwise manner;
i.e., we will optimize $q_{\text{train}},\theta,\sigma^{2}$ separately
in each iteration. More details will be given in the following
subsections. 

\subsection{Minimizing with respect to $q_{\text{train}}$ and $\sigma^{2}$ }

With all the other parameters fixed, the sub-problem to minimize with
respect to $\sigma^{2}$ is 
\begin{align}
\underset{\sigma^{2}\in\mathbb{R}_{>0}}{\arg\min\ } & \frac{n}{2}\log\sigma^{2}+\frac{1}{q_{\text{train}}-1}\log(1+\left(\frac{2}{q_{\text{train}}-1}-n_{\text{train}}\right)^{-1}\cdot\sigma^{-2}\label{eq:subproblem-sigma2}\\
 & \quad\cdot\left(\left\langle y_{\text{train}}-\mathbf{X}_{\text{train}}\theta,\Psi_{\text{train}}^{-1}\left(y_{\text{train}}-\mathbf{X}_{\text{train}}\theta\right)\right\rangle +2n_{\text{train}}\sum_{j=2}^{p+1}w\left(\theta_{j}\right)\right))\nonumber 
\end{align}
which has a smooth objective function with respect to $\sigma^{2}$.
The first-order optimality condition 
\begin{align*}
\frac{n}{2}= & \frac{1}{q_{\text{train}}-1}\\
 & \cdot\frac{\left(\frac{2}{q_{\text{train}}-1}-n_{\text{train}}\right)^{-1}\cdot\left(\left\langle y_{\text{train}}-\mathbf{X}_{\text{train}}\theta,\Psi_{\text{train}}^{-1}\left(y_{\text{train}}-\mathbf{X}_{\text{train}}\theta\right)\right\rangle +2n_{\text{train}}\sum_{j=2}^{p+1}w\left(\theta_{j}\right)\right)}{\sigma^{2}+\left(\frac{2}{q_{\text{train}}-1}-n_{\text{train}}\right)^{-1}\cdot\left(\left\langle y_{\text{train}}-\mathbf{X}_{\text{train}}\theta,\Psi_{\text{train}}^{-1}\left(y_{\text{train}}-\mathbf{X}_{\text{train}}\theta\right)\right\rangle +2n_{\text{train}}\sum_{j=2}^{p+1}w\left(\theta_{j}\right)\right)}
\end{align*}
implies that the optimal value for the subproblem \eqref{eq:subproblem-sigma2} takes
minimizer 
\begin{align}
\overline{\sigma^{2}} & =\left(\frac{1}{q_{\text{train}}-1}/\frac{n}{2}-1\right)\cdot\left(\frac{2}{q_{\text{train}}-1}-n_{\text{train}}\right)^{-1}\label{eq:sigma2-closed-form-minimizer}\\
 & \quad\cdot\left(\left\langle y_{\text{train}}-\mathbf{X}_{\text{train}}\theta,\Psi_{\text{train}}^{-1}\left(y_{\text{train}}-\mathbf{X}_{\text{train}}\theta\right)\right\rangle +2n_{\text{train}}\sum_{j=2}^{p+1}w\left(\theta_{j}\right)\right)>0,\nonumber 
\end{align}
which is feasible. The feasible set for $q_{\text{train}}$ is $\left(1,1+\frac{2}{n_{\text{train}}}\right)$,
in this view, when $n_{\text{train}}$ is large, the numerical stability
will be an issue if minimization is carried out with respect
to $q_{\text{train}}$ directly. Thus, we choose to minimize with
respect to $\frac{1}{q_{\text{train}}-1}\in\left(\frac{n_{\text{train}}}{2},\infty\right)$. 

First of all, we are to prove that such minimization is feasible. 
\begin{lem}
The objective function \eqref{eq:opt-problem} has a local minimizer
in $\left(\frac{n_{\text{train}}}{2},\infty\right)$ with respect
to $\frac{1}{q_{\text{train}}-1}$. 
\end{lem}

\begin{proof}
Since the objective function \eqref{eq:opt-problem} is continuous
and smooth with respect to $\frac{1}{q_{\text{train}}-1}$, we only
need to analyze the derivative when $\frac{1}{q_{\text{train}}-1}\searrow0$
and $\frac{1}{q_{\text{train}}-1}\rightarrow\infty$.

$\frac{1}{q_{\text{train}}-1}\rightarrow\infty$:

Stirling's formula states that 
\begin{equation}
\lim_{x\rightarrow\infty}\frac{\Gamma\left(x\right)}{\sqrt{\frac{2\pi}{x}}\left(\frac{x}{e}\right)^{x}\left(1+O\left(x^{-1}\right)\right)}=1.
\end{equation}
Thus, 
\begin{equation}
\lim_{\frac{1}{q_{\text{train}}-1}\rightarrow\infty}\frac{\Gamma\left(\frac{1}{q_{\text{train}}-1}\right)}{\Gamma\left(\frac{1}{q_{\text{train}}-1}-\frac{n_{\text{train}}}{2}\right)}\cdot\left(\frac{2}{q_{\text{train}}-1}-n_{\text{train}}\right)^{-\frac{n_{\text{train}}}{2}}=1
\end{equation}
then 
\begin{equation}
\lim_{\frac{1}{q_{\text{train}}-1}\rightarrow\infty}-\log\left(\frac{\Gamma\left(\frac{1}{q_{\text{train}}-1}\right)}{\Gamma\left(\frac{1}{q_{\text{train}}-1}-\frac{n_{\text{train}}}{2}\right)}\cdot\left(\frac{2}{q_{\text{train}}-1}-n_{\text{train}}\right)^{-\frac{n_{\text{train}}}{2}}\right)=0.
\end{equation}
We also have 
\begin{align*}
 & \frac{1}{q_{\text{train}}-1}\log(1+\left(\frac{2}{q_{\text{train}}-1}-n_{\text{train}}\right)^{-1}\cdot\sigma^{-2}\\
 & \quad\cdot\left(\left\langle y_{\text{train}}-\mathbf{X}_{\text{train}}\theta,\Psi_{\text{train}}^{-1}\left(y_{\text{train}}-\mathbf{X}_{\text{train}}\theta\right)\right\rangle +2n_{\text{train}}\sum_{j=2}^{p+1}w\left(\theta_{j}\right)\right))\\
 & =O\left(\left(\frac{1}{q_{\text{train}}-1}\right)/\log\left(\frac{1}{q_{\text{train}}-1}\right)\right),
\end{align*}
which implies that this term will goes to infinity as $\frac{1}{q_{\text{train}}-1}\rightarrow\infty$.
Thus, the objective function \eqref{eq:opt-problem} goes to infinity
as $\frac{1}{q_{\text{train}}-1}\rightarrow\infty$. 

$\frac{1}{q_{\text{train}}-1}\searrow\frac{n_{\text{train}}}{2}$: 

Since $\Gamma\left(\frac{1}{q_{\text{train}}-1}-\frac{n_{\text{train}}}{2}\right)\rightarrow\infty$
as $\frac{1}{q_{\text{train}}-1}\searrow\frac{n_{\text{train}}}{2}$. 

The penalized log-likelihood involving $\frac{1}{q_{\text{train}}-1}$
can be simplified as 
\begin{align*}
 & -\log\frac{\Gamma\left(\frac{1}{q_{\text{train}}-1}\right)}{\Gamma\left(\frac{1}{q_{\text{train}}-1}-\frac{n_{\text{train}}}{2}\right)}\cdot(\left(\frac{2}{q_{\text{train}}-1}-n_{\text{train}}\right)\\
 & \qquad+\left\langle y_{\text{train}}-\mathbf{X}_{\text{train}}\theta,\left(\sigma^{2}\Psi_{\text{train}}\right)^{-1}\left(y_{\text{train}}-\mathbf{X}_{\text{train}}\theta\right)\right\rangle +2n_{\text{train}}\sum_{j=2}^{p+1}w\left(\theta_{j}\right))^{\frac{1}{1-q_{\text{train}}}}\rightarrow\infty
\end{align*}
as $\frac{1}{q_{\text{train}}-1}\searrow\frac{n_{\text{train}}}{2}$.
Thus, the subproblem to minimize with respect to $\frac{1}{q_{\text{train}}-1}$
is coercive on $\left(\frac{n_{\text{train}}}{2},\infty\right)$.
Coercivity implies that any minimizing sequence $\left\{ \left(\frac{1}{q_{\text{train}}-1}\right)_{j}\right\} $
must be contained within a bounded subset of $\left(\frac{n_{\text{train}}}{2},\infty\right)$.
Thus, Bolzano--Weierstrass theorem implies the existence of a convergent
subsequence. Let $\left\{ \left(\frac{1}{q_{\text{train}}-1}\right)_{j_{k}}\right\} $
be one such subsequence, and let $\overline{\left(\frac{1}{q_{\text{train}}-1}\right)}$
be its limit. Since the subproblem has a continuous objective function
with respect to $\frac{1}{q_{\text{train}}-1}$, the objective function
is lower--semicontinuous and the value of the objective function at
$\overline{\left(\frac{1}{q_{\text{train}}-1}\right)}$ is less than or equal to
the value of the objective function at $\left(\frac{1}{q_{\text{train}}-1}\right)_{j_{k}}$
for all $k=1,2,\dots,\infty$. Thus, since $\left\{ \left(\frac{1}{q_{\text{train}}-1}\right)_{j}\right\} $
is a minimizing sequence, the value of the objective function at $\overline{\left(\frac{1}{q_{\text{train}}-1}\right)}$
is less than or equal to the infimum of the objective function on
$\left(\frac{n_{\text{train}}}{2},\infty\right)$. Hence, since the
entire minimizing sequence is contained in $\left(\frac{n_{\text{train}}}{2},\infty\right)$,
$\overline{\left(\frac{1}{q_{\text{train}}-1}\right)}\in\left(\frac{n_{\text{train}}}{2},\infty\right)$
solves the subproblem of minimizing with respect to $\frac{1}{q_{\text{train}}-1}$. 

Unlike $\sigma^{2}$, the minimizer for $\frac{1}{q_{\text{train}}-1}$
is not in closed form, and the evaluation of the derivative with respect
to $\frac{1}{q_{\text{train}}-1}$ can not be carried out efficiently.
Thus, we apply Brent's line-search method to optimize the $\frac{1}{q_{\text{train}}-1}$
subproblem \cite{Brent1971}. 
\end{proof}

\subsection{Minimizing with respect to $\theta$ \label{subsec:optimization-motivation}}

Minimizing with respect to $\theta$ involves a nonconvex smooth function
and a convex nonsmooth function, which is termed a composite problem. In 
Section \ref{subsec:Proximal-HZ}, we developed a proximal conjugate
gradient algorithm for such composite optimization. 

In this part, we will establish an important remark regarding the
Oracle penalty. The subproblem we are to minimize with respect to
$\theta$ is 
\begin{align}
 & \underset{\theta\in\mathbb{R}^{p+1}}{\arg\min\ }\left\langle y_{\text{train}}-\mathbf{X}_{\text{train}}\theta,\Psi_{\text{train}}^{-1}\left(y_{\text{train}}-\mathbf{X}_{\text{train}}\theta\right)\right\rangle +2n_{\text{train}}\sum_{j=2}^{p+1}w\left(\theta_{j}\right)\nonumber \\
\Leftrightarrow & \underset{\theta\in\mathbb{R}^{p+1}}{\arg\min\ }\frac{1}{2n_{\text{train}}}\left\langle y_{\text{train}}-\mathbf{X}_{\text{train}}\theta,\Psi_{\text{train}}^{-1}\left(y_{\text{train}}-\mathbf{X}_{\text{train}}\theta\right)\right\rangle +\sum_{j=2}^{p+1}w\left(\theta_{j}\right)\nonumber \\
\Leftrightarrow & \underset{\theta\in\mathbb{R}^{p+1}}{\arg\min\ }\frac{1}{2n_{\text{train}}}\left\langle \theta,\mathbf{X}_{\text{train}}^{T}\Psi_{\text{train}}^{-1}\mathbf{X}_{\text{train}}\theta\right\rangle -2\left\langle y_{\text{train}},\Psi_{\text{train}}^{-1}\mathbf{X}_{\text{train}}\theta\right\rangle +\sum_{j=2}^{p+1}w\left(\theta_{j}\right)\label{eq:central-trend-subproblem}
\end{align}
$w$ can be chosen as oracle penalties such as SCAD/MCP penalties.
And it has been shown that both SCAD / MCP penalties admit a difference-of-convex
decomposition to a first-order smooth concave term plus $\lambda$
times $\ell_{1}$ penalty. The quadratic loss function is clearly
convex and smooth. This justifies our assumption for the objective
function. To carry out the proximal Hager-Zhang conjugate gradient
method proposed in Section \ref{subsec:Proximal-HZ}, we need to calculate
$L_{\nabla g}$, the $L-$smoothness constant for the smooth component.
Previous work suggests $L_{\nabla g}=\max\ \left\{ \text{max eigenvalue of }\frac{1}{n_{\text{train}}}\mathbf{X}_{\text{train}}^{T}\Psi_{\text{train}}^{-1}\mathbf{X}_{\text{train}},c_{\text{penalty}}\right\} $,
where $c_{\text{penalty}}$ is the $L-$smoothness constant for the
smooth component of the penalty, which will be $\frac{1}{a-1}$ for
SCAD and $\frac{1}{\gamma}$ for MCP \cite{Yang2024}. 
\begin{rem}
For high dimensional data, often times, the number of covariates exceeds
the number of observations; i.e, $\text{null}\left(\mathbf{X}_{\text{train}}\right)\neq\emptyset$.
Both SCAD/MCP penalties take constant values in $B_{\infty}\left(0,c\right)$\footnote{$B_{\infty}\left(0,c\right)$ denotes the open ball in uniform norm
\emph{in the corresponding space}, centered at the origin with radius
$c$. }; where $c=a\lambda$ for SCAD and $c=\gamma\lambda$ for MCP. Given
any stationary point $\bar{\theta}$ in the nonempty solution set
defined by $\mathbf{X}_{\text{train}}^{T}\Psi_{\text{train}}^{-1}\mathbf{X}_{\text{train}}-\mathbf{X}_{\text{train}}^{T}y_{\text{train}}=0$.
For the set $\bar{\theta}+\text{null}\left(\mathbf{X}_{\text{train}}\right)\setminus B_{\infty}\left(0,c\right)$,
each point in the relative interior (which is nonempty) of this set
is a Clarke stationary point, which implies that any algorithm with
a starting point in this set will converge in $0$ steps. This might
pose an issue for signal recovery, since $\text{null}\left(\mathbf{X}_{\text{train}}\right)$
is a vector subspace and some points can be very far from the origin. 
\end{rem}

\begin{rem}
In view of the subproblem with respect to $\theta$, it is
trivial that the minimizer for \eqref{eq:central-trend-subproblem}
does not depend on the other parameters, which are $q$ and $\sigma^{2}$.
Since the $q$Gaussian distribution is a generalization for all bell curve distributions, the estimation of the central trend
using the maximum likelihood principle for bell curve distributions is
equivalent to minimize a quadratic function, which has a breakdown point
of $0$. 

Taking into account the optimization subproblem with respect to $\theta$, it is evident that
the solution to \eqref{eq:central-trend-subproblem} remains unaffected
by the other parameters, namely $q$ and $\sigma^{2}$. Given that the $q$Gaussian distribution
extends the framework of bell curve distributions, the problem \eqref{eq:central-trend-subproblem} implies that estimating the central trend through
the maximum likelihood principle for all bell curve distributions is equivalent
to minimizing a quadratic function of the central trend, therefore characterized by a breakdown point
of $0$. 
\end{rem}

\subsection{Prediction for $y_{\text{test}}$ }

To show how prediction can be made, we will show the methods to
predict $y_{\text{test}}$ using the trained model in this subsection.
The same method applies for validation when predictions on $y_{\text{val}}$
are needed or to predict any new data based on the trained model. Since
the data are mutually $q$Gaussian, \eqref{eq:recover-q} implies
that 
\begin{equation}
\frac{1}{q_{\text{train}}-1}-n_{\text{train}}=\frac{1}{q_{\text{val}}-1}-n_{\text{val}}=\frac{1}{q_{\text{test}}-1}-n_{\text{test}}=\frac{1}{q-1}-n,
\end{equation}
which will be used to recover the value of the shape parameter $q_{\text{val}},q_{\text{test}}$.
Note that when $n_{\text{new}}$ data points are introduced, the total
number of observations $n$ changes from $n_{\text{train}}+n_{\text{val}}+n_{\text{test}}$
to $n_{\text{train}}+n_{\text{val}}+n_{\text{test}}+n_{\text{new}}$,
thus, the value of $q_{\cdot}$ will change for the entire dataset.
However, $q_{\text{train}}$ stays the same; thus, we suggest inferring the shape parameter for each data set based on $q_{\text{train}}$
directly using the equation above. With $q_{\cdot}$ calculated, it
is straightforward to estimate the $q-$variance-covariance matrix
$\mathbb{E}_{q}\left[\left(y_{\cdot}-\mathbf{X}\theta\right)\left(y_{\cdot}-\mathbf{X}\theta\right)^{T}\right]$
based on \eqref{eq:q-variance-covariance}; or, \emph{if existing,}
the variance-covariance matrix $\mathbb{E}\left[\left(y_{\cdot}-\mathbf{X}\theta\right)\left(y_{\cdot}-\mathbf{X}\theta\right)^{T}\right]$
based on \eqref{eq:variance-covariance}. 



\section{\label{sec:Conclusion-and-Discussion} Conclusion and Discussion }

This paper explores the field of statistical sparse learning, focusing
on modeling correlated data through the lens of maximizing Tsallis
entropy. It addresses the limitations inherent in the conventional
Gaussian distribution, notably its lack of robustness towards outliers
and underlying shape assumptions, by advocating for the $q$Gaussian
distribution. This distribution, derived from Tsallis entropy maximization,
represents a novel approach to handling correlated data and heterogeneity
--- elements frequently encountered in biostatistical contexts involving
genetic and longitudinal studies. 

This paper encompasses a re-derived probability density function
for the multivariate $q$Gaussian distribution based on Tsallis entropy
maximization. Statistical modeling based on the derived density paves the
way for the analysis of correlated data and heterogeneity and enables
variable selection. Furthermore, we have developed an innovative framework
capable of converting any numerical method, originally designed to
identify equilibria in flows, into a tool for tackling composite
optimization problems that are prevalent in statistical sparse learning.
By applying this framework to the Hager-Zhang conjugate gradient algorithm,
we have crafted an effective and stable algorithm tailored to the
challenges of sparse statistical learning. Given the abundance of
methods for numerically identifying equilibria for globally Lipschitz
flows, our approach significantly broadens the arsenal of techniques
available to address sparse statistical learning optimization challenges. 

In conclusion, our research positions the $q$Gaussian distribution,
underpinned by maximizing Tsallis entropy, as a robust and adaptable
alternative to Gaussian-based methodologies in statistical sparse
learning on correlated data. This breakthrough not only confronts
the traditional limitation of Gaussian assumptions, but also paves
the way for expanded investigation into Tsallis entropy-maximizing
distributions, particularly within the domain of biostatistics and
allied disciplines. 

Future directions for research include the exploration of the log-linear
model through the lens of Tsallis entropy maximization, akin to approaches
previously based on Shannon's entropy. Moreover, the study of the
phenomenon called \emph{volatility smirk} in financial return data may benefit
from employing the log-$q$Gaussian distribution ---- a transformation
of the $q$Gaussian distribution, which can provide deeper insights
into the nuances of financial markets. Additionally, in the field
of statistical computing research, our framework that transforms numerical
methods for identifying flow equilibria into algorithms for solving
composite optimization problems opens numerous avenues for future
research, especially in a sparse learning context.

\section{Acknowledgments}
This work was supported by the ISM Scholarship for Outstanding PhD Candidates awarded to K. Yang, the NSERC Discovery Grant to C. Greenwood (Grant Number: RGPIN-2019-04482), the NSERC Discovery Grant to M. Asgharian (Grant Number: RGPIN-2024-05640), and the CANSSI Collaborative Research Team Grant to C. Greenwood and G. Cohen Freue.

\newpage{}

\printbibliography

\end{document}